\documentclass[final]{siamltex}

\usepackage{graphicx}   
\usepackage{multirow}
\usepackage{amsmath}
\usepackage{amsfonts}
\usepackage{latexsym}
\usepackage{verbatim}
\usepackage{amsmath}    

\newtheorem{remark}[theorem]{Remark}

\newcommand{\beq}{\begin{equation}}
\newcommand{\eeq}{\end{equation}}
\newcommand{\beqa}{\begin{eqnarray}}
\newcommand{\eeqa}{\end{eqnarray}}
\newcommand{\beqas}{\begin{eqnarray*}}
\newcommand{\eeqas}{\end{eqnarray*}}
\newcommand{\bi}{\begin{itemize}}
\newcommand{\ei}{\end{itemize}}
\newcommand{\ba}{\begin{array}}
\newcommand{\ea}{\end{array}}

\newcommand{\nn}{\nonumber}

\def\eqnok#1{(\ref{#1})}

\def\vgap{\vspace*{.1in}}

\def\setU{{X}}
\def\hu{{u^*}}

\def\tu{{\tilde x}}

\newcommand{\bbr}{\Bbb{R}}
\def\w{\omega}

\def\VI{{\rm VI}}

\title{
On the Convergence Properties of Non-Euclidean\\
Extragradient Methods for Variational Inequalities with\\
 Generalized Monotone Operators
\thanks{The author of this paper was partially supported by
    NSF grants CMMI-1000347 and DMS-1319050,
    ONR grant N00014-13-1-0036, and
    NSF CAREER Award CMMI-1254446.}
}

\date{April 12, 2012}
\author{
     Cong D. Dang
    \thanks{Department of Industrial and Systems
    Engineering, University of Florida, Gainesville, FL, 32611.
    (email: {\tt congdd@ufl.edu}).}
    \and
    Guanghui Lan
    \thanks{Department of Industrial and Systems
    Engineering, University of Florida, Gainesville, FL, 32611.
    (email: {\tt glan@ise.ufl.edu}). }
}

\begin{document}

\maketitle

\begin{abstract}
In this paper, we study a class of generalized monotone variational inequality (GMVI) problems
whose operators are not necessarily monotone (e.g., pseudo-monotone).
We present non-Euclidean extragradient (N-EG) methods for computing 
approximate strong solutions of these problems, and demonstrate how their iteration complexities
depend on the global Lipschitz or H\"{o}lder continuity properties for their operators and the smoothness 
properties for the distance generating function used in the N-EG algorithms. We also 
introduce a variant of this algorithm by incorporating a simple line-search procedure 
to deal with problems with more general continuous operators. Numerical studies are conducted
to illustrate the significant advantages of the developed algorithms
over the existing ones for solving large-scale GMVI problems.

\vspace{.1in}

\noindent {\bf Keywords:} Complexity, Monotone variational inequality,
Pseudo-monotone variational inequality, Extragradient methods,
Non-Euclidean methods, Prox-mapping

\end{abstract}

\vspace{0.1cm}
\setcounter{equation}{0}
\section{Introduction} \label{sec_intro}
Variational inequality (VI) has been widely studied in the
literature due to its encompassing power of describing a wide range
of optimization, equilibrium and complementarity problems (see
\cite{FacPang03} and references therein). Given an
nonempty closed convex set $X\subseteq \bbr^n$ and a continuous
mapping $F:X \to \bbr^n$, the variational inequality problem,
denoted by $\VI(X,F)$, is to find $x^* \in X$ satisfying
\beq 
\label{def_VI} \langle F(x^*), x - x^* \rangle \ge 0 \ \ \ \forall
\, x \in X.
\eeq
Such a point $x^*$ is often called a {\sl strong
solution} of $\VI(X,F)$.

The extragradient method, initially proposed by
Korpelevich~\cite{Kor76}, is a classical method for solving VI
problems. It improves the usual gradient projection method
(e.g.,~\cite{Sibony70,Bert99}) by performing an additional metric
projection step at each iteration. While earlier studies on
extragradient methods were focused on their asymptotical convergence
analysis (see, e.g., \cite{SolSva99-1,Sun95-1,tseng00-1}), much
recent effort has been directed to the complexity analysis of these
types of methods. In particular, Nemirovski~\cite{Nem05-1} presented
a generalized version of Korpelevich's extragradient method and
analyzed its iteration complexity in terms of the computation of a
{\sl weak solution}, i.e., a point $x^* \in X$ such that \beq
\label{def_weak} \langle F(x), x - x^* \rangle \ge 0 \ \ \ \forall
\, x \in X. \eeq Note that if $F(\cdot)$ is monotone and continuous,
a weak solution of $\VI(X,F)$ must be a strong solution and vise
versa. Moreover, he showed that one can possibly improve the
complexity results by employing the non-Euclidean projection
(prox-mapping) steps (see \eqnok{s29}) in place of the two metric
projection steps in Korpelevich's extragradient method. These types of methods
are referred to as {\sl non-Euclidean extragradient (N-EG)} methods
in this paper. Similar results have also been developed by Auslender and
Teboulle~\cite{AuTe05-1} for their interior projection methods
applied to monotone VI problems. More recently, Monteiro and
Svaiter~\cite{MonSva09-1} established the complexity for a class of
hybrid proximal extragradient methods ~\cite{SS99-1} which covers
Korpelevich's extragradient method as a special case.
Other approaches and their associated rate of convergence
for solving VI problems have also been studied (see, e.g., \cite{Nest07-2,KoNeSh11-1}).
Note that in all these previous studies in \cite{Nem05-1,AuTe05-1,MonSva09-1,Nest07-2,KoNeSh11-1},
the operator $F(\cdot)$ is assumed to be monotone.

In this paper, we consider a more general class of VI problems for which the
operator $F(\cdot)$ is not necessarily monotone. In particular, we make the
following much weaker assumption about
the monotonicity of $\VI(X,F)$, i.e., relation \eqnok{def_weak}
holds for any strong solution $x^* \in X$ (e.g., \cite{HaPan90-1}).
This class of VI problems, refereed to as {\sl generalized monotone
variational inequalities} (GMVI), cover both monotone and
pseudo-monotone VI problems. It is also not difficult to construct
GMVI problems whose operators are neither monotone nor
pseudo-monotone (c.f., \eqnok{simple_ex}). However, to the best of our
knowledge, there does not exist any complexity results for the
extragradient methods applied to GMVI problems in the literature. In particular, 
the previous complexity studies conducted for VI~\cite{Nem05-1,AuTe05-1,MonSva09-1,Nest07-2,KoNeSh11-1} relies 
on the monotonicity assumption of the operator $F(\cdot)$ and hence 
are not applicable for the possibly non-monotone 
GMVI problems. Moreover, most of the previous complexity analysis has been
conducted for computing a weak solution approximately satisfying
\eqnok{def_weak}, and there exists very few complexity results for computing
approximate strong solutions (see~\cite{MonSva09-1}).
In particular, if our goal is to compute an approximate strong solution, 
it was unclear, even for
monotone VI problems, how the N-EG method will be more advantageous
over Korpelevich's extragradient method.

The main goal of this paper is to present a generalization of
the N-EG method in~\cite{Nem05-1}
for solving GMVI problems and discuss its convergence properties
in terms of the computation of approximate strong solutions.
Our major contributions are summarized as follows.
Firstly, we present a new termination criterion based on the residual function
associated with the prox-mapping, and discuss its relations with a few other
possible notions of approximate strong solutions to $\VI(X,F)$.
In particular, we show that under certain conditions, if a point $x\in X$
has a small residual, it must be associated with an approximate strong solution
$y \in X$ with a small optimality gap $g(y)$, where 
\[
g(y) = \max_{z \in X}
\langle F(y), y - z\rangle.
\]
We also show how this termination criterion is related to the notion of
an approximate strong solution recently proposed by Monteiro and Svaiter~\cite{MonSva09-1}
to deal with VI problems with unbounded $X$.

Secondly, we study the complexity of the N-EG method for solving GMVI
problems whose operator $F(\cdot)$ satisfies certain
global continuity assumptions.
In particular, by employing a novel analysis, we show that, if $F(\cdot)$ is
Lipschitz continuous, then the N-EG method applied to GMVI problems can generate
a solution $y_k \in X$ with $g(y_k)$ bounded by
\[
{\cal O}(1) \left(\frac{L (\alpha + {\cal Q}) \Omega^2_{\w,X}}{\alpha \sqrt{k}}\right).
\]
Here, ${\cal O}(1)$ denotes an absolute constant,
$L$ is the Lipschitz constant of $F$, $\alpha$ and
${\cal Q}$ are certain constants of the distance generating function $\w(\cdot)$
used to define the prox-mapping, and $\Omega^2_{\w,X}$ is
a characteristic constant depending on $\w(\cdot)$ and $X$.
We also 
consider GMVI problem with H\"{o}lder continuous
operators and show that the N-EG method possesses an
${\cal O}(1/k^{\nu/2})$ rate of convergence for solving this class of problems,
where $\nu \in (0,1]$ denotes the level of
continuity.
Our development also improves an existing result for computing an approximate strong
solution for Lipschitz and monotone VI with unbounded $X$ by removing
a purification procedure introduced by Monteiro and Svaiter in \cite{MonSva10-3}.

Thirdly, in order to deal with more general GMVI problems whose operators 
are not necessarily H\"{o}lder continuous, we present a variant of
N-EG method by incorporating a simple line-search procedure (N-EG-LS) and show that
it can generate a sequence of solutions converging to
a strong solution of $\VI(X,F)$.
It should be noted that, while earlier extragradient type methods
for GMVI problems with a general continuous operator
(e.g., Solodov and Svaiter~\cite{SolSva99-1}, Sun~\cite{Sun95-1}) rely on a
certain monotonicity property of the metric
projection (e.g., Gafni and Bertsekas \cite{GafBer84-1}),
such a property is not assumed by the prox-mapping in general.
We present certain sufficient conditions on the prox-mapping 
which can guarantee the convergence of the N-EG-LS algorithm.
More specifically, we show that these conditions are satisfied by
the prox-mapping induced by distance generating functions 
with Lipschitz continuous gradients.

Finally, we present promising numerical results for the developed
N-EG methods for solving GMVI problems. In particular, we
demonstrate that the N-EG-LS method can be more advantageous over
N-EG method if the Lipschitz constant is big or unknown. Moreover,
we show that the N-EG methods with a properly chosen distance
generating function $\w(\cdot)$ can outperform the
Euclidean methods especially when the dimension
$n$ is big.

This paper is organized as follows. In Section~\ref{sec_problem},
we describe in more details the GMVI problems. We discuss
the prox-mapping and the termination criterion associated with
the prox-mapping in Section~\ref{sec_term}. Then we present
the N-EG method for solving GMVI problems with Lipschitz or
 H\"{o}lder continuous operators in Section~\ref{sec_basic_alg}.
 Section~\ref{sec_adv_alg} is devoted to the N-EG-LS method
 applied to GMVI problems with general continuous operators.
 Finally, numerical results are presented in Section~\ref{sec_num}.

\vgap

\subsection{Notation and terminology} \label{notation-sco}
let $\bbr^n$ be an arbitrary finite dimensional vector space endowed with
the inner product $\langle \cdot, \cdot \rangle$, $\|\cdot\|$ denote
a norm in $\bbr^n$ (not necessarily the one associated with the inner product),
and $\|\cdot\|_*$ denotes its conjugate norm. Let $X \subseteq \bbr^n$ be closed
and convex. A function $f: X \to \bbr$ is said to have $L$-Lipschitz
continuous gradient if it is differentiable and
\[
\|\nabla f(x) - \nabla f(y)\|_* \le L \|x - y\|, \forall x, y \in X.
\]

For a given $m \times n$ real-valued matrix $A$, letting $\|A\|_2$ be the spectral
norm and $\|A\|_{\max} = \max_{ij}\{|A_{ij}|\}$, we have
\beq \label{rel_norms}
\|A\|_{\max} \le \|A\|_2 \le \sqrt{mn} \|A\|_{\max}.
\eeq

We use $\mathbb{N}$ to denote the set of natural numbers.

\section{The problem of interest} \label{sec_problem}
Given an nonempty closed convex set $X\subseteq \bbr^n$ and
a continuous mapping $F:X \to \bbr^n$, the problem of interest
in this paper is find a strong solution $x^*$ of $\VI(X,F)$,
i.e., a vector $x^* \in X$ such that \eqnok{def_VI}
holds.
In this paper, we assume that the solution set $X^*$ of $\VI(X,F)$
is nonempty. Moreover, the following assumption is made throughout the paper.

\vgap

\noindent {\bf{A1}} For any
$x^* \in X^*$ we have
\beq \label{monotone}
\langle F(x), x - x^* \rangle \ge 0 \ \ \ \forall \, x \in X.
\eeq

\vgap

Clearly, Assumption A1 is satisfied if $F(\cdot)$ is monotone, i.e.,
\beq \label{mon_true} \langle F(x) - F(y), x - y \rangle \ge 0 \ \ \
\forall \, x, y \in X. \eeq Moreover, this assumption holds if
$F(\cdot)$ is pseudo-monotone, i.e., \beq \label{mon_true_p} \langle
F(y), x - y \rangle \ge 0  \ \ \Longrightarrow \ \ \langle F(x), x -
y \rangle \ge 0. \eeq As an example, $F(\cdot)$ is pseudo-monotone if
it is the gradient of a real-valued differentiable pseudo-convex
function. It is also not difficult to construct VI problems that
satisfy \eqnok{monotone}, but their operator $F(\cdot)$ is neither
monotone nor pseudo-monotone anywhere (see \cite{SolSva99-1} and
references therein). One set of simple examples are given by all the
functions $F: \bbr \to \bbr$ satisfying \beq \label{simple_ex} F(x)
\left\{
\begin{array}{ll}
= 0, & x = x_0;\\
\ge 0, & x \ge x_0;\\
\le 0,& x \le x_0.
\end{array}
\right. \eeq These problems, although satisfying Assumption A1 with
$x^* = x_0$, can be neither monotone nor pseudo-monotone.

For future reference, we say that $\VI(X,F)$ is a generalized monotone VI (GMVI) problem
whenever Assumption A1 is satisfied.

The condition given by \eqnok{def_VI} is the standard definition of
a {\sl strong solution} to $\VI(X,F)$. Recall that
a solution $x^*$ satisfying \eqnok{monotone} is
usually called is a {\sl weak solution} to $\VI(X,F)$.
Clearly, under our assumption,  a strong solution
must be a weak solution for the GMVI problems. The inverse is also true 
if $F(\cdot)$ is continuous and monotone. However, such a relation does not 
necessarily hold when $F(\cdot)$ is not
monotone and hence the computation of an approximate weak solution is not particularly
useful in these cases. In addition, as pointed out by Monteiro and
Svaiter \cite{MonSva09-1}, a strong solution to $\VI(X,F)$ admits
certain natural explanations for some important classes of monotone VI
problems, e.g., the complementarity problems. This paper focuses on
the computation of approximate strong solutions to $\VI(X,F)$ (see
Section~\ref{sec_term}).

Depending on the continuity properties of $F(\cdot)$,
we consider the following four different classes of VI problems.
\begin{itemize}
\item [i)] $F(\cdot)$ is Lipschitz continuous:
\beq \label{smooth}
\|F(x) - F(y)\|_* \le L \|x -y\|, \ \ \ \forall \, x, y \in X;
\eeq
\item [ii)] $F(\cdot)$ is H\"{o}lder continuous: for some $\nu \in (0,1]$,
\beq \label{smooth1}
\|F(x) - F(y)\|_* \le L \|x-y\|^\nu, \ \ \ \forall \, x, y \in X;
\eeq
\item [iii)] $F(\cdot)$ is locally Lipschitz continuous:
for every $x \in X$, there exists a neighborhood $B_x$ of $x$, such that
\beq \label{smooth2}
\|F(x) - F(y)\|_* \le L \|x - y\|, \ \ \ \forall \, x, y \in B_x \subset X;
\eeq
\item [iv)] $F(\cdot)$ is continuous:
\beq \label{smooth3}
\lim_{y \to x} \|F(x) - F(y)\|_* = 0, \ \ \ \forall x \in X.
\eeq
\end{itemize}
Clearly, if $F(\cdot)$ is Lipschitz continuous, then it is
H\"{o}lder continuous with $\nu = 1$. Moreover, in view of the
assumption that $\nu >0$, a H\"{o}lder continuous $F(\cdot)$ must be
continuous but not vise versa. In addition, a locally Lipschitz
continuous $F(\cdot)$ must be continuous but the inverse is not
necessarily true. In Sections~\ref{sec_basic_alg} and
\ref{sec_adv_alg},
we will present algorithms for solving different classes of VI problems
and show how their convergence properties depend on the continuity assumption of $F(\cdot)$.

\section{Prox-mapping and termination criteria} \label{sec_term}
In this section, we discuss the main computational construct, i.e., the prox-mapping,
that will be used in the non-Euclidean extragradient methods. We also present a termination criterion based
on the prox-mapping and show how it relates to some other termination criteria
for solving VI problems. It is worth noting that the results in this section
does not require Assumption A1.

\subsection{Distance generating function and prox-mapping}
We review the concept of prox-mapping (e.g., \cite{Nem05-1,AuTe06-1,NJLS09-1}) in this subsection.

A function $\w:\,X\to \bbr$ is called a {\em distance generating function} modulus $\alpha>0$
with respect to $\|\cdot\|$, if the following conditions hold: i) $\w(\cdot)$ is convex and
continuous on $X$; ii) the set
\[
\begin{array}{ll}
X^o=\left \{x\in X:\partial w(x) \neq \emptyset \right \}
\end{array}
\]
is convex (note that $X^o$ always contains the relative interior of $X$); and
iii) restricted to $X^o$, $\w(\cdot)$ is continuously
differentiable and strongly convex with parameter
$\alpha$ with respect to  $\|\cdot\|$, i.e.,
\begin{equation}\label{s27}
\langle \nabla \w(x')-\nabla\w(x), x'-x \rangle\geq\alpha
\|x'-x\|^2,\;\;\forall x',x\in X^o .
\end{equation}

Given a distance generating function $\w$,
the {\em prox-function} $V:X^o\times X\to\bbr_+$ is defined by
\begin{equation}\label{s450}
V(x,z)=\w(z)-[\w(x)+ \langle \nabla \w(x), z-x \rangle].
\end{equation}
The function $V(\cdot,\cdot)$ is also called the Bregman's distance,
which was initially studied by Bregman \cite{Breg67} and later by many others
(see \cite{AuTe06-1,BBC03-1,Kiw97-1,Teb97-1} and references therein).
In this paper, we assume that the prox-function
$V(x,z)$ is chosen such that, for a given $x \in X^o$, the {\em prox-mapping} $P_x: \bbr^n \to \bbr^n$
defined as
\begin{equation}\label{s29}
P_x(\phi)=\arg\min_{z\in X}\big\{\langle \phi, z \rangle+V(x,z)\big\}
\end{equation}
is easily computable.
It can be seen from the strong convexity of $\w(\cdot)$
and \eqnok{s450} that
\beq \label{bnd_v}
 V(x,z) \ge \frac{\alpha}{2} \|x - z\|^2, \ \ \ \forall x, z \in X.
\eeq
In some cases, we assume that the distance generating function $\w(\cdot)$ satisfies
\beq \label{smoothomega}
\|\nabla \w(x) - \nabla \w(z)\|_* \le {\cal Q} \|x-z\|, \ \ \forall x, z \in X,
\eeq
for some ${\cal Q} \in (0,\infty)$. Under this assumption, it can be easily seen that (see, e.g.,
Lemma 1.2.3 of \cite{Nest04})
\beq \label{quad_growth}
V(x, z) \le \frac{{\cal Q}}{2} \|x - z\|^2, \ \ \ \forall x, z \in X.
\eeq
We say that the prox-function $V(\cdot, \cdot)$ is growing
quadratically whenever condition \eqnok{quad_growth} holds.





\vgap

Proposition \ref{examples} below provides a few examples for the selection of $\|\cdot\|$ and
distance generating function $\w(\cdot)$. 
More such examples can be found, for example, in \cite{AuTe06-1,BenNem00,nemyud:83}.

\begin{proposition} \label{examples}
\begin{itemize}
\item [a)] If $X = \bbr^n$, $\|\cdot\| = \|\cdot\|_2$ and $\w(x) = \|x\|_2^2/2$, then we have
$\alpha = {\cal Q}=1$.
\item [b)] If $X = \{ x \in \bbr^n: \sum_{i=1}^n x_i = 1, x_i \ge 0, i = 1, \ldots, n\}$, $\|\cdot\| = \|\cdot\|_1$
and $\w(x) = \sum_{i=1}^n (x_i+\delta/n) \log (x_i+\delta/n)$ with $\delta = 10^{-16}$,
then we have $\alpha = {\cal O}(1)$ and ${\cal Q} = 1 + n /\delta$. Here ${\cal O}(1)$
denotes an absolute constant.
\item [c)] If $X = \{x \in \bbr^n: \|x\|_1 \le 1\}$, where
$\|x\|_1 = \sum_{i=1}^n |x_i|$, $\|\cdot\| = \|\cdot\|_1$, and
\beq \label{lpnorm}
\w(x) = \frac{1}{2}\|x\|_p^2 = \frac{1}{2} \left(\sum_{i=1}^n |x_i|^p \right)^\frac{2}{p}
\eeq
with $p = 1 + 1 / \ln n$, then we have $\alpha = {\cal O}(1)(1/\ln n)$.
\end{itemize}
\end{proposition}

\begin{proof}
Part a) is obvious and part b) has been shown in Chapter 5 of \cite{BenNem00}.
Moreover, the strong convexity of $\w(x)$ in \eqnok{lpnorm} and the estimation
of its modulus (with $p = 1 + 1/ \ln n$) is shown in \cite{nemyud:83}.
\end{proof}


\vgap

The distance generating function $\w(\cdot)$ also gives rise to the following
characteristic entity that will be used frequently in our convergence analysis:
\beq \label{def_D}
D_{\w,X} := \sqrt{\max_{x \in X}\w(x)-\min_{x\in X} \w(x)}.
\eeq
Let $x_1$ be the minimizer of $\w$ over $X$. Observe that $x_1\in
X^o$, whence $\nabla w(x_1)$ is well defined and satisfies
$\langle \nabla\omega(x_1), x-x_1 \rangle \geq 0$ for all $x\in X$,
which combined with the strong convexity of $\w$ implies that
\begin{equation}\label{strong_h_e}
\frac{\alpha}{ 2}\|x-x_1\|^2\leq V(x_1,x)\leq
\omega(x)-\omega(x_1)\leq D_{\w,X}^2,\;\;\forall x\in X,
\end{equation}
and hence
\begin{equation}\label{dist_x}
\|x-x_1\|\leq \Omega_{\w,X}:=\sqrt{\frac{2}{\alpha}}D_{\w,X} \, \,
\mbox{and} \, \,
\|x - x'\| \le 2 \Omega_{\w,X}, \,\,\forall x, x' \in X.
\end{equation}

\subsection{A termination criterion based on the prox-mapping} \label{sec_def_res}
In this subsection, we introduce a termination criterion for solving VI
associated with the prox-mapping.

We first provide a simple characterization of a strong solution
to $\VI(X,F)$.

\begin{lemma} \label{optcond}
A point $x \in X$ is a strong solution of $\VI(X,F)$ if and only if
\beq \label{fixed}
x = P_{x}(\gamma F(x))
\eeq
for some $\gamma > 0$.
\end{lemma}

\begin{proof}
If \eqnok{fixed} holds, then
by the optimality condition of \eqnok{s29}, i.e.,
\beq \label{optimality}
\langle \gamma F(x) + \nabla \w (P_{x}( \gamma F(x))) - \nabla \w( x), z - P_{x}(\gamma F(x)) \rangle \ge 0,
\ \ \ \forall z \in X,
\eeq
we have $\langle \gamma F(x), z - x \rangle \ge 0$ for any $z \in X$, which, in view of
the fact that $\gamma > 0$ and definition \eqnok{def_VI}, implies
that $x$ is a strong solution of $\VI(X,F)$. The ``only if'' part of the statement
easily follows from the optimality condition of \eqnok{s29}.
\end{proof}

\vgap

Motivated by Lemma \ref{optcond}, we can define the
residual function for a given $x \in X$ as follows.

\vgap

\begin{definition} \label{def_solution}
Let $\|\cdot\|$ be a given norm in $\bbr^n$, $\w(\cdot)$ be a distance generating function
modulus $\alpha > 0$ w.r.t. $\|\cdot\|$ and $P_x(\cdot)$
be the prox-mapping defined in \eqnok{s29}. Then, for
some positive constant $\gamma$, we define the residual
$R_\gamma(\cdot)$ at the point $x \in X$ as
\beq \label{def_res}
R_\gamma(x) :=
\frac{1}{\gamma} \left[ x - P_{x}(\gamma F(x))
\right].
\eeq
\end{definition}

Observe that in the Euclidean setup where $\|\cdot\| = \|\cdot\|_2$ and
$\w(x) = \|x\|_2^2 /2$, the residual $R_\gamma(\cdot)$ in \eqnok{def_res} reduces to
\beq \label{res_eu}
R_\gamma(x) = \frac{1}{\gamma}
\left[x - \Pi_X(x - \gamma F(x)) \right],
\eeq
where $\Pi_X(\cdot)$ denotes the metric projection over $X$.
Such a residual function in \eqnok{res_eu} has been used
in the asymptotic analysis of different algorithms for solving VI problems (e.g.,
\cite{SolSva99-1,Sun95-1}). In particular, if $F(\cdot)$ is the gradient of
a real-valued differentiable function $f(\cdot)$, the residual $R_\gamma(\cdot)$ in
\eqnok{res_eu} corresponds to the well-known projected gradient of $f(\cdot)$ at
$x$ (see, e.g., \cite{LanMon08-1,LanMon09-1,Nest04}).

The following two results are immediate consequences of
Lemma~\ref{optcond} and Definition~\ref{def_solution}.

\begin{lemma} A point $x\in X$ is a strong solution of $\VI(X,F)$ if and only if
$\|R_\gamma(x)\| = 0$ for some $\gamma > 0$;
\end{lemma}

\begin{lemma} \label{asy_con}
Suppose that $x_k \in X$ and $\gamma_k \in (0, \infty)$, $k = 1,2, \ldots$,
satisfy the following conditions:
\begin{itemize}
\item [i)] $\lim_{k\to \infty} V(x_k, P_{x_k}(\gamma_k F(x_k))) =0$;
\item [ii)] There exists $K \in \mathbb{N}$ and $\gamma^* > 0$ such that
$\gamma_k \ge \gamma^*$ for any $k \ge K$.
\end{itemize}
Then we have $\lim_{k \to \infty} \|R_{\gamma_k}(x_k)\| = 0$. If in
addition, the sequence $\{x_k\}$ is bounded, there exists an
accumulation point $\tilde x$ of $\{x_k\}$ such that $\tilde x \in
X^*$, where $X^*$ denotes the solution set of $\VI(X,F)$.
\end{lemma}

\begin{proof}
Denote $y_k = P_{x_k}(\gamma_k F(x_k))$. It follows from
\eqnok{bnd_v} and condition i) that $\lim_{k \to \infty} \|x_k -
y_k\| = 0$. This observation, in view of Condition ii) and
Definition \ref{def_solution}, then implies that $\lim_{k\to \infty}
\|R_{\gamma_k}(x_k)\| = 0$. Moreover, if $\{x_k\}$ is bounded, there
exist a subsequence $\{\tilde x_i\}$ of $\{x_k\}$ obtained by
setting $\tilde x_i = x_{n_i}$ for $n_1 \le n_2 \le \ldots$, such
that $\lim_{i \to \infty} \|\tilde x_i - \tilde x\| = 0$. Let
$\{\tilde{y}_i\}$ be the corresponding subsequence in $\{y_k\}$,
i.e., $y_i = P_{x_{n_i}}(\gamma_{n_i} F(x_{n_i}))$, and $\tilde
\gamma_i = \gamma_{n_i}$. We have $\lim_{i \to \infty} \|\tilde x_i
- \tilde y_i\| = 0$. Moreover, by \eqnok{optimality}, we have
\[
\langle F(\tilde x_i) + \frac{1}{\tilde \gamma_i} \left[ \nabla \w(\tilde y_i) - \nabla \w(\tilde x_i) \right], z - \tilde y_i \rangle
\ge 0,
\ \ \forall z \in X, \, \forall i \ge 1.
\]
Tending $i$ to $+\infty$ in the above inequality, and using the continuity of $F(\cdot)$ and $\nabla \w(\cdot)$, and condition ii), we
conclude that $ \langle F(\tilde x), z - \tilde x \rangle \ge 0$ for any $z \in X$.
\end{proof}

\vgap

In the remaining part of this section, we relate the
residual $R_\gamma(\cdot)$ to a few other possible
termination criteria for solving $\VI(X,F)$.

Observe that, if the set $X$ is bounded, then
in view of definition \eqnok{def_VI}, one can measure the
inaccuracy of a solution $x \in X$ by the gap function
(see \cite{Hearn82-1,HaPan90-1} and references therein):
\beq \label{def_gap1}
g(x) := \sup_{z \in X} \langle F(x), x - z \rangle.
\eeq
It can be easily seen that $g(x) \ge 0$ for any $x \in X$
and that the point $x^* \in X$ is a strong solution
of $\VI(X,F)$ if and only if $g(x^*) = 0$.
Note also that the gap function in \eqnok{def_gap1} does not depend on
any algorithmic parameters, while the definition of the residual function
$R_\gamma(\cdot)$ in \eqnok{def_res}
depends on the selection of $\|\cdot\|$ and $\w(\cdot)$.



However, if $X$ is unbounded, then the gap function $g(\cdot)$ in
\eqnok{def_gap1} may not be well-defined.
To address this issue, Monteiro and Svaiter~\cite{MonSva09-1} suggested a generalization
of the gap function $g(\cdot)$ so as to deal with the unbounded feasible sets.
More specifically, they define
a new gap function $\tilde{g}(\cdot, \phi)$ as
\beq \label{def_gap3}
\tilde{g}(x, \phi) := \sup_{z \in X} \langle F(x) + \phi, x -z\rangle.
\eeq
Observe that there always exists a point $\phi$
(e.g., $\phi = - F(x)$) such that $\tilde{g}(x, \phi)$ is well-defined
for any $x \in X$. Accordingly, an approximate strong solution of $\VI(X,F)$
can be defined as follows.

\begin{definition} \label{def_solution1}
A point $x \in X$ is called an $(\epsilon, \delta)$-strong solution of
$\VI(X,F)$, if there exists some $\phi$ such that $\|\phi\|_* \le \epsilon$ and
$\tilde{g}(x,\phi) \le \delta$.
\end{definition}

\vgap

Observe that the above definition of an $(\epsilon, \delta)$-strong solution
of $\VI(X,F)$ relies on the selection of the norm $\|\cdot\|_*$ (and hence $\|\cdot\|$).
In~\cite{MonSva09-1}, Monteiro and Svaiter
focused on the Euclidean setup where $\|\cdot\| = \|\cdot\|_2$
and considered VI problems with monotone and Lipschitz continuous operator $F(\cdot)$.
Moreover, as pointed out in \cite{MonSva09-1},
the two tolerances, namely $\epsilon$ and $\delta$,
used in the definition of an $(\epsilon, \delta)$-strong
solution possess natural interpretations in the context of
complementarity problems. In particular, the following result
has been shown in Proposition 3.1 of \cite{MonSva09-1}.

\begin{proposition} \label{from_MS}
Assume that $X = K$, where $K$ is a nonempty closed convex cone
(and $K^*$ be its dual cone). Then, $x \in K$ is an $(\epsilon, \delta)$-strong solution of
$\VI(X,F)$ if and only if there exists $\phi \in K^*$ such that
\[
\|F(x) - \phi \|_* \le \epsilon, \ \ \
\langle x, \phi \rangle \le \delta.
\]
\end{proposition}

\vgap

In other words, the first tolerance $\epsilon$ measures the infeasibility
of $F(x)$ with respect to the dual cone while the second tolerance
$\delta$ measures the size of the complementarity slackness.

\vgap
In the following result we show some relations between the the gap functions defined in \eqnok{def_gap1} and
\eqnok{def_gap3} and the residual function $R_\gamma(\cdot)$ defined in \eqnok{def_res}.

\begin{proposition} \label{prop_relation}
Let $x \in X$ be given. Also, for some $\gamma > 0$, let us denote
\beqa
 x^+ &:=& P_x(\gamma F(x)), \label{def_xplus}\\
\phi_{\gamma}(x) &:=& F(x) - F(x^+) + \frac{1}{\gamma}\left[ \nabla \w(x^+) - \nabla \w(x)\right] \label{rel_errors}.
\eeqa
\begin{itemize}
\item [a)] We have $\tilde{g}(x^+,\phi_{\gamma}(x)) \le 0$ for any $x \in X$ and $\gamma > 0$;
\item [b)] If $\w(\cdot)$ has ${\cal Q}$-Lipschitz continuous gradients
w.r.t. $\|\cdot\|$ and $F(\cdot)$ is H\"{o}lder continuous, i.e., condition \eqnok{smooth1} holds
for some $\nu \in (0,1]$, then
\beq \label{bnd_gap0}
\|\phi_{\gamma}(x)\|_* \le L \left[\gamma \|R_\gamma(x)\|\right]^\nu + {\cal Q} \|R_\gamma(x)\|;
\eeq
\item [c)] If, in addition, the set $X$ is bounded, then
\beq \label{bnd_gap}
g(x^+) \le 2  \Omega_{\w,X} \left[L \gamma^\nu \|R_\gamma(x)\|^\nu
+ {\cal Q} \|R_\gamma(x)\| \right],
\eeq
where $\Omega_{\w,X}$ is defined in \eqnok{dist_x}.
\end{itemize}
\end{proposition}

\begin{proof}
Using the definition of $x^+$ in \eqnok{def_xplus}, the optimality condition of \eqnok{s29}
and the fact that  $\nabla V(x,z) = \nabla \w(z) - \nabla \w(x)$,
we have
\[
\langle F(x) + \frac{1}{\gamma}\left[\nabla \w(x^+) - \nabla \w(x) \right],
\, z - x^+ \rangle \ge 0, \ \ \forall \, z \in X,
\]
which together with \eqnok{def_gap3} and \eqnok{rel_errors} then imply that
\beqa
\tilde g(x^+,\phi_{\gamma}(x))
&=& \sup_{z \in X} \langle F(x^+) + \phi_{\gamma}(x), \, x^+ - z \rangle \nn\\
&=& \sup_{z \in X} \langle  F(x) + \frac{1}{\gamma}\left[\nabla \w(x^+) - \nabla \w(x) \right],
\, x^+ - z \rangle \le 0. \label{tmp_bnd_gap}
\eeqa
We have thus shown part a). Now it follows from
the triangular inequality, \eqnok{rel_errors}, \eqnok{smoothomega}, \eqnok{smooth1}
and \eqnok{def_res} that
\beqas
\|\phi_{\gamma}(x)\|_* &\le& \|F(x) - F(x^+)\|_* + \frac{1}{\gamma}\|\nabla \w(x^+) - \nabla \w(x)\|_*\\
&\le& L \|x- x^+\|^\nu + \frac{{\cal Q}}{\gamma} \|x^+ - x\|
= L \left[\gamma \|R_\gamma(x)\|\right]^\nu + {\cal Q} \|R_\gamma(x)\|,
\eeqas
which implies \eqnok{bnd_gap0}. By using
the above conclusion, \eqnok{def_gap1} and \eqnok{tmp_bnd_gap}, we have
\beqas
g(x^+) &=& \sup_{z \in X} \langle [F(x^+) + \phi_{\gamma}(x)] - \phi_{\gamma}(x), x^+ - z \rangle\\
&\le& \sup_{z \in X} \langle - \phi_{\gamma}(x), x^+ - z \rangle \\
&\le&  \left[L \left(\gamma \|R_\gamma(x)\|\right)^\nu + {\cal Q} \|R_\gamma(x)\| \right] \,
\sup_{z \in X} \|x^+ - z\|,
\eeqas
The above inequality together with \eqnok{dist_x} then imply \eqnok{bnd_gap}.
\end{proof}

\vgap

By using Proposition~\ref{prop_relation}, we can easily see the relation between
the residual function $R_\gamma(\cdot)$
and the notion of an $(\epsilon, \delta)$-strong solution under the
Euclidean setup.

\begin{corollary} \label{cor_lip_rel}
Suppose that $\|\cdot\| = \|\cdot\|_2$ and $\w(x) = \|\cdot\|_2^2/2$, and
also assume that $F(\cdot)$ is Lipschitz continuous (i.e., condition \eqnok{smooth} holds). Then,
if, for some $x \in X$ and $\gamma > 0$, the point $x^+$ given by \eqnok{def_xplus}
satisfies $\|R_\gamma(x)\| \le \epsilon$,
it must be an $\left((L\gamma + 1) \epsilon, 0\right)$-strong solution.
In particular, if $\gamma \le 1/L$, then the point $x^+$ must be
an $(2 \epsilon, 0)$-strong solution of $\VI(X,F)$.
\end{corollary}

\begin{proof}
The result directly follows from Proposition~\ref{prop_relation}.b) and the facts that $\nu = 1$
and ${\cal Q} = 1$.
\end{proof}

\vgap

It is interesting to note that under the conditions given in
Corollary \ref{cor_lip_rel}, if $\gamma < 1/L$ and $\|R_\gamma(x)\|
\le \epsilon/2$ then the solution given \eqnok{def_xplus} must be an
$(\epsilon, 0)$-strong solution. Such a solution is stronger than an
$(\epsilon, \delta)$-strong solution with $\delta > 0$ as it does
not depend on the second tolerance $\delta$. For example, in the
context of complementarity problems, the complementarity slackness
constraint will be satisfied exactly by an $(\epsilon, 0)$-strong
solution.

\setcounter{equation}{0}

\section{VI problems with Lipschitz or H\"{o}lder continuous operators} \label{sec_basic_alg}
Our main goal in this section is to establish the complexity
of a non-Euclidean extragradient (N-EG) method for solving the GMVI problems
discussed in Section~\ref{sec_problem}.
We assume throughout this section that the operator $F(\cdot)$ is
either Lipschitz or, more generally, H\"{o}lder continuous.

The extragradient method is a classical method for solving VI problems that was initially
proposed by Korpelevich~\cite{Kor76}.
While earlier studies
on Korpelevich's extragradient method or
its variants were focused on their asymptotical convergence behaviour
(see, e.g., \cite{SolSva99-1,Sun95-1,tseng00-1}),
the complexity analysis of these types of methods has only appeared recently
in the literature~\cite{Nem05-1,MonSva09-1}. More specifically, Nemirovski~\cite{Nem05-1}
established the complexity of a generalized version of Korpelevich's extragradient method
for computing a weak solution of $\VI(X,F)$ and showed that one can possibly improve its
performance by replacing the projection step with the prox-mapping defined in \eqnok{s29}.
Some of these results were generalized by Auslender and Teboulle~\cite{AuTe05-1}
in their interior projection methods for monotone variational inequalities.
More recently, Monteiro and Svaiter~\cite{MonSva09-1}
studied the complexity of the original Korpelevich's extragradient method
(under a more general framework).
Most of these previous studies need to assume the operator $F(\cdot)$ to be monotone and
Lipschitz continuous. To the best of our knowledge, the complexity
of extragradient-type methods for solving more general VI problems
(e.g., the operator $F(\cdot)$ is pseudo-monotone and has different levels of continuity)
has never been studied in the literature. It is worth noting that under this general setting,
the notion of a weak solution is not useful any more and one has to resort to
the notions of approximate strong solutions as discussed in Section~\ref{sec_term}.
Moreover, it is unclear how one can benefit from taking
the prox-mapping (rather than metric projection)
in the extragradient method for the computation of strong solutions to $\VI(X,F)$.

\vgap



\noindent{\bf The non-Euclidean extragradient (N-EG) method for GMVI:}
\begin{itemize}
\item [] {\bf Input:} Initial point $x_1 \in X$ and stepsizes $\{\gamma_k\}_{k \ge 1}$.
\item [0)] Set $k = 1$.
\item [1)] Compute
\beqa
y_k &=& P_{x_{k}}(\gamma_k F(x_{k})), \label{tt1}\\
x_{k+1} &=& P_{x_{k}}(\gamma_k F(y_k)). \label{tt2}
\eeqa
\item [2)] Set $k = k+1$ and go to Step 1.
\end{itemize}

We now add a few remarks about the above N-EG method. Firstly,
observe that under the Euclidean case when $\|\cdot\| = \|\cdot\|_2$
and $\w(x) = \|x\|^2/2$, the computation of $(y_t, x_t)$, $t \ge 1$,
is the same as Korpelevich's extragradient or Euclidean
extragradient (E-EG) method. 
Secondly, it should be noted that, while the above N-EG method is
similar to Nemirovski's mirror-prox method for solving monotone VI
problems, its convergence analysis and the specification of the
algorithmic parameters (e.g., $\|\cdot\|$, $\w(\cdot)$ and
$\gamma_k$) differ significantly from those in \cite{Nem05-1}, since
we are dealing with a much wider class of problems and using different
termination criteria.

\vgap

In order to establish the convergence properties of the above N-EG method,
we first need to show a few technical results.

Let $p(u)$ be a convex function over a convex set $\setU \in \bbr^n$.
Assume that $\hu$ is an optimal solution of the problem
$
\min \{ p(u) + \|u-\tu\|^2: u \in \setU \}
$
for some $\tu \in \setU$. Due to the well-known fact that
the sum of a convex and a strongly convex function is also strongly convex,
one can easily see that
\[
p(u) + \|u-\tu\|^2 \ge \min \{ p(v) + \|v-\tu\|^2: v \in \setU \} +
\|u - \hu\|^2.
\]
The next lemma generalizes this result to the case where the
function $\|u - \tu\|^2$ is replaced with the prox-function
$V(\tu,u)$ associated with a convex function $\w$.
It can be viewed as a Bregman version of ``growth formula" for
strongly convex functions and is based on a Pythagora like
formula for Bregman distances. The proof of this result
can be found, e.g., in Lemma 1 of \cite{Lan10-3} and
Lemma 6 of \cite{LaLuMo11-1}.

\begin{lemma} \label{tech1_prox}
Let $\setU$ be a convex set in $\bbr^n$ and
$p, \w:\setU \to \bbr$ be differentiable convex functions. Assume that
$\hu$ is an optimal solution of
$\min \{ p(u) + V(\tu, u) : u \in \setU\}$.
Then,
\[
p(\hu) +  V(\tu, \hu) + V(\hu, u)
 \le p(u) + V(\tu, u) ,
\ \ \forall u \in \setU.
\]
\end{lemma}

With this result, we can show an important recursion of
the N-EG method for $\VI(X,F)$.
More specifically, let $x^* \in X^*$ be an optimal solution,
the next result describes how the
the distance $V(x_k, x^*)$ decreases at each iteration of
the N-EG method.

\begin{lemma}
Let $x_{1}\in X$ be given and the pair $(y_k,x_{k+1}) \in X \times X$ be computed
according to \eqnok{tt1}-\eqnok{tt2}. Also let $X^*$
denote the solution set of $\VI(X,F)$. Then,
the following statements hold:
\begin{itemize}
\item [a)] There exists $x^* \in X^*$ such that
\beq \label{recur}
- \frac{\gamma_k^2}{2 \alpha} \|F(x_k) - F(y_k)\|_*^2 + V(x_{k}, y_k) \le
V(x_{k}, x^*) - V(x_{k+1}, x^*);
\eeq
\item [b)] If $F(\cdot)$ is H\"{o}lder continuous (i.e., condition \eqnok{smooth1} holds),
then there exists $x^* \in X^*$ such that for any $\nu \in (0,1]$,
\beq \label{recur_h}
V(x_k, y_k) - 2^{\nu - 1} \, L^2 \gamma_k^2 \alpha^{-(1+\nu)} \left[ V(x_k, y_k)\right]^\nu
\le V(x_k, x^*) - V(x_{k+1}, x^*).
\eeq
In particular, if $F(\cdot)$ is Lipschitz continuous (i.e., condition \eqnok{smooth} holds),
then we have
\beq \label{recur_l}
\left(1- L^2 \gamma_k^2 \alpha^{-2} \right) V(x_k, y_k) \le V(x_k, x^*) - V(x_{k+1}, x^*).
\eeq
\end{itemize}
\end{lemma}

\begin{proof}
We first show part a). By \eqnok{tt1} and Lemma~\ref{tech1_prox} (with $p(\cdot) = \gamma_k
\langle F(x_k), \cdot \rangle$, $\tilde x = x_k$ and $\hu = y_k$), we have
\[
\gamma_k \langle F(x_{k}), y_k - x\rangle
+ V(x_k,y_k) + V(y_k, x) \le V(x_k, x), \ \forall x \in X.
\]
Letting $x = x_{k+1}$ in the above inequality, we obtain
\beq \label{bnd_dist}
\gamma_k \langle F(x_k), y_k - x_{k+1}\rangle
+ V(x_k,y_k) + V(y_k, x_{k+1}) \le V(x_k, x_{k+1})
\eeq
Moreover, by \eqnok{tt2} and Lemma~\ref{tech1_prox} (with $p(\cdot) = \gamma_k
\langle F(y_k), \cdot \rangle$, $\tilde x = x_k$ and $\hu = x_{k+1}$), we have
\[
\gamma_k \langle F(y_k), x_{k+1} - x \rangle
+ V(x_k, x_{k+1}) + V(x_{k+1}, x) \le V(x_k,x), \ \forall x \in X.
\]
Replacing $V(x_k, x_{k+1})$ in the above inequality with the bound in \eqnok{bnd_dist}
and noting that $\langle F(y_k), x_{k+1} - x \rangle = \langle F(y_k), y_k - x \rangle
- \langle F(y_k), y_k - x_{k+1} \rangle$, we have
\[
\gamma_k \langle F(y_k), y_k - x \rangle + \gamma_k \langle F(x_k) - F(y_k), y_k - x_{k+1}\rangle
+ V(x_k, y_k) + V(y_k, x_{k+1}) + V(x_{k+1}, x) \le V(x_k, x),
\]
which, in view of Assumption A1, then implies that
\beq \label{reg_recursion}
\gamma_k \langle F(x_k) - F(y_k), y_k - x_{k+1}\rangle
+ V(x_k, y_k) + V(y_k, x_{k+1}) + V(x_{k+1}, x^*) \le V(x_k, x^*),
\eeq
In order to show \eqnok{recur}, we only need to
bound the left hand side of \eqnok{reg_recursion}. By using Cauchy Schwarz inequality and \eqnok{bnd_v}, we have
\beqas
\lefteqn{\gamma_k \langle F(x_k) - F(y_k), y_k - x_{k+1}\rangle
+ V(x_{k}, y_k) + V(y_k, x_{k+1}) }\\
&\ge& - \gamma_k \|F(x_k) - F(y_k)\|_* \|y_k - x_{k+1}\|
+ V(x_{k}, y_k) + V(y_k, x_{k+1}) \\
&\ge& - \gamma_k \|F(x_k) - F(y_k)\|_* \left[\frac{2}{\alpha}
V(y_k, x_{k+1})\right]^\frac{1}{2} + V(x_{k}, y_k) + V(y_k, x_{k+1}) \\
&\ge& - \frac{\gamma_k^2}{2 \alpha} \|F(x_k) - F(y_k)\|_*^2 + V(x_{k}, y_k),
\eeqas
where the last inequality follows from Young's inequality.
Combining the above observation with \eqnok{reg_recursion},
we arrive at relation \eqnok{recur}.
Now, it follows from the assumption \eqnok{smooth1} and \eqnok{bnd_v} that
\[
\|F(x_k) - F(y_k)\|_*^2 \le L^2 \|x_k - y_k\|^{2\nu}
\le L^2 \left[\frac{2}{\alpha} V(y_k,x_k)\right]^\nu.
\]
Combining the previous observation with \eqnok{recur}, we obtain \eqnok{recur_h}.
Relation \eqnok{recur_l} immediately follows from \eqnok{recur_h}
with $\nu = 1$.
\end{proof}

\vgap

We are now ready to establish the complexity of the N-EG method
for solving GMVI problems.
We start with the relatively easier case when $F(\cdot)$ is
Lipschitz continuous.

\begin{theorem} \label{theorem_lip}
Suppose that $F(\cdot)$ is Lipschitz continuous (i.e.,
condition \eqnok{smooth} holds) and that the stepsizes $\gamma_k$
are set to
\beq \label{stepsize_l}
\gamma_k = \frac{\alpha}{\sqrt{2} L}, \ \ \ k \ge 1.
\eeq
Also let $R_\gamma(\cdot)$, $g(\cdot)$, $\tilde{g}(\cdot, \cdot)$
and $\phi_\gamma(\cdot)$ be defined in \eqnok{def_res}, \eqnok{def_gap1},
\eqnok{def_gap3} and \eqnok{rel_errors}, respectively.
\begin{itemize}
\item [a)] For any $k \in \mathbb{N}$, there exists $i \le k$ such that
\beq \label{iter_bnd_l}
\|R_{\gamma_i}(x_i)\|^2
\le \frac{8 L^2}{\alpha^3 k} V(x_1, x^*), \ \ \ k \ge 1;
\eeq
\item [b)] If $\w(\cdot)$ has ${\cal Q}$-Lipschitz continuous gradients w.r.t. $\|\cdot\|$, then
for every $k \in \mathbb{N}$, there exists $i \le k$ such that
\beq \label{iter_bnd_ms}
\tilde{g}(y_i, \phi_{\gamma_i}(x_i)) \le 0 \ \ \
\mbox{and} \ \ \
\|\phi_{\gamma_i}(x_i)\|_* \le \frac{2 L (\alpha + \sqrt{2} {\cal Q})}{\alpha^{3/2}}
\sqrt{\frac{V(x_1, x^*)}{k}};
\eeq
\item [c)] If, in addition, the set $X$ is bounded, then for every $k \in \mathbb{N}$,
there exists $i \le k$ such that
\beq
g(y_i) \le \frac{2\sqrt{2} L \,(\alpha + \sqrt{2} {\cal Q}) \, \Omega_{\w,X}^2 }{\alpha \sqrt{k}},
\eeq
where $\Omega_{\w,X}$ is defined in \eqnok{dist_x}.
\end{itemize}
\end{theorem}

\begin{proof}
Using \eqnok{recur_l} and \eqnok{stepsize_l}, we have
\[
\frac{1}{2} V(x_k, y_k) \le V(x_k, x^*) - V(x_{k+1}, x^*), \ \ k \ge 1.
\]
Also it follows from \eqnok{bnd_v} and definition \eqnok{def_res} that
\beq \label{bnd_v_below}
V(x_k, y_k) \ge \frac{\alpha}{2} \|x_k - y_k\|^2
= \frac{\alpha \gamma_k^2}{2} \|R_{\gamma_k}(x_k)\|^2.
\eeq
Combining the above two observations, we obtain
\[
\gamma_k^2 \|R_{\gamma_k}(x_k)\|^2 \le \frac{4}{\alpha} \left[ V(x_k, x^*) - V(x_{k+1}, x^*)\right],
k \ge 1.
\]
By summing up these inequalities we arrive at
\[
\sum_{i=1}^k \gamma_i^2 \min_{i = 1, \ldots, k} \|R_{\gamma_i}(x_i)\|^2
\le \sum_{i=1}^k  \gamma_i^2 \|R_{\gamma_i}(x_i)\|^2 \le \frac{4}{\alpha} V(x_1, x^*), \ \
k \ge 1,
\]
which implies that
\beq \label{bnd_L}
\min_{i = 1, \ldots, k} \|R_{\gamma_i}(x_i)\|^2 \le \frac{4}{\alpha \sum_{i=1}^k \gamma_i^2}
V(x_1, x^*).
\eeq
Using the above inequality and \eqnok{stepsize_l}, we obtain
the bound in \eqnok{iter_bnd_l}.
Part b) directly follows from Proposition~\ref{prop_relation}.a) and b) and
bound \eqnok{iter_bnd_l}. Moreover, Part c) follows from Proposition.c),
bound \eqnok{iter_bnd_l} and the definition of $\Omega_{X,\w}$ in \eqnok{dist_x}.
\end{proof}

\vgap

We now add a few comments about the results obtained in
Theorem~\ref{theorem_lip}.
Firstly, in view of Theorem~\ref{theorem_lip}.b),
under the Euclidean setup where $\|\cdot\| = \|\cdot\|_2$ and
$\w(x) = \|x\|_2^2/2$ (hence $\alpha = 1$), the error bounds obtained in \eqnok{iter_bnd_ms}
is stronger than the corresponding ones in Theorem 5.2 of \cite{MonSva09-1}
in the following sense:
(a) the bound is applicable to a more general class of VI problems, i.e., the GMVI problems;
and (b) the second residual $\delta$ used in the notion of
an $(\epsilon, \delta)$-strong solution associated with the sequence
$\{y_k\}$ always vanishes as opposed to the one obtained in \cite{MonSva09-1}.
Secondly, our results apply also to the non-Euclidean setup where $\w(\cdot)$ is not necessarily
$\|\cdot\|_2^2/2$. This can lead to more efficient variants of the N-EG
method for solving large-scale GMVI problems as demonstrated in Section 6. 

In the next result, we discuss the convergence properties of the
N-EG method for solving GMVI problems with H\"{o}lder continuous operators.
For the sake of simplicity, we assume here that the number of iterations $k$
is fixed a priori.

\vgap

\begin{theorem} \label{theorem_hold}
Suppose that $F(\cdot)$ is H\"{o}lder continuous (i.e.,
condition \eqnok{smooth1} holds) and that the stepsizes
$\gamma_i, i = 1, \ldots, k$, in the N-EG method are set to
\beq \label{stepsize_h}
\gamma_i = \frac{\alpha^\frac{1+\nu}{2}}{L (2 \nu)^{\frac{\nu}{2}}}
\left(\frac{1}{k}\right)^{\frac{1-\nu}{2}},
\eeq
where $k$ be the number of iterations given a priori.
Also let $R_\gamma(\cdot)$, $g(\cdot)$, $\tilde{g}(\cdot, \cdot)$
and $\phi_\gamma(\cdot)$ be defined in \eqnok{def_res}, \eqnok{def_gap1},
\eqnok{def_gap3} and \eqnok{rel_errors}, respectively.
\begin{itemize}
\item [a)] There exists $i \le k$ such that
\beq \label{iter_bnd_h}
\|R_{\gamma_i}(x_i)\|^2 \le \frac{8 L^2}{\alpha^{2+\nu} k^\nu} [1 + V(x_1, x^*)];
\eeq
\item [b)] If $\w(\cdot)$ has ${\cal Q}$-Lipschitz continuous gradients w.r.t. $\|\cdot\|$,
then there exists $i \le k$ such that
\beq \label{iter_bnd_ms_h}
\tilde{g}(y_i, \phi_{\gamma_i}(x_i)) \le 0 \ \ \
\mbox{and} \ \ \
\|\phi_{\gamma_i}(x_i)\|_* \le \frac{2\sqrt{2} C_\alpha L }{(\alpha k)^\frac{\nu}{2}} [1+V(x_1,x^*)]^\frac{1}{2},
\eeq
where $C_\alpha = 4 + {\cal Q}/\alpha$;
\item [c)] If, in addition, the set $X$ is bounded, then there exists $i \le k$ such that
\beq \label{iter_bnd_gap_h}
g(y_i) \le \frac{4\sqrt{2} C_\alpha L  \Omega_{\w,X}}{(\alpha k)^\frac{\nu}{2}} [1+V(x_1,x^*)]^\frac{1}{2},
\eeq
where $\Omega_{\w,X}$ is defined in \eqnok{dist_x}.
\end{itemize}
\end{theorem}

\begin{proof}
Since the case when $\nu = 1$ has been shown in Theorem~\ref{theorem_lip},
we assume that $\nu \in (0,1)$.
After rearranging the terms in \eqnok{recur_h},
we obtain, $\forall i = 1, \ldots, k$,
\beq \label{rel_holder}
\frac{1}{2} V(x_i, y_i) \le V(x_i, x^*) - V(x_{i+1}, x^*) +
\underbrace{2^{\nu - 1}
\, L^2 \gamma_i^2 \alpha^{-(1+\nu)} \left[ V(x_i, y_i)
\right]^\nu - \frac{1}{2} V(x_i, y_i)}_{\Delta_i}.
\eeq
Note that, in view of the observation that
\[
\max_{d \ge 0} \left\{ a \, d^\nu -  d/2 \right\}
\le a (2 \nu a)^{-\frac{\nu}{\nu-1}} = \alpha^{-\frac{1}{\nu-1}} (2\nu)^{-\frac{\nu}{\nu-1}}, \ \ \
\forall \, a > 0, \, \nu \in (0,1),
\]
and relation \eqnok{stepsize_h}, we have
\beq \label{bnd_delta}
\Delta_i \le \frac{1}{2} \left[ L^2 \gamma_i^2 \alpha^{-(1+\nu)}\right]^{-\frac{1}{\nu-1}}
(2 \nu)^{-\frac{\nu}{\nu-1}} = \frac{1}{2k} \le \frac{1}{k}, \ \ i = 1, \ldots, k.
\eeq
Moreover, it follows from \eqnok{bnd_v_below} and
\eqnok{stepsize_h} that
\[
V(x_i, y_i) \ge \frac{\alpha \gamma_i^2}{2} \|R_{\gamma_i}(x_i)\|^2
= \frac{\alpha^{2+\nu} k^{\nu-1}}{2 L^2 (2 \nu)^\nu}
\|R_{\gamma_i}(x_i)\|^2.
\]
Using the previous two bounds in \eqnok{rel_holder}, we conclude
\beq \label{import_bnd_h}
\frac{\alpha^{2+\nu} k^{\nu-1}}{4 L^2 (2 \nu)^\nu}
\|R_{\gamma_i}(x_i)\|^2 \le V(x_i, x^*) - V(x_{i+1}, x^*)
+ \frac{1}{k}, \ \ \ i = 1, \ldots, k.
\eeq
Summing up the above inequalities and using the fact that
$k \min_{i=1, \ldots, k} \|R_{\gamma_i}(x_i)\|^2 \le
\sum_{i=1}^k \|R_{\gamma_i}(x_i)\|^2$,
we have
\[
\frac{\alpha^{2+\nu} k^\nu}{4 L^2 (2 \nu)^\nu}
\min_{i=1, \ldots, k} \|R_{\gamma_i}(x_i)\|^2
\le V(x_1,x^*) - V(x_{N+1}, x^*)
+ 1 \le V(x_1,x^*)+1,
\]
which clearly implies \eqnok{iter_bnd_h} in view of
the fact that $(2\nu)^\nu \le 2$.

We now show part b). The first inequality of \eqnok{iter_bnd_ms_h} follows
directly from Proposition~\ref{prop_relation}.a) and the definitions
of $x_i$ and $y_i$. Note that by \eqnok{bnd_gap0} and the fact that $\gamma_1= \gamma_2 = \ldots =\gamma_k$ due to
\eqnok{stepsize_h}, we have
\beqas
\min_{i=1,\ldots, k} \|\phi_{\gamma_i}(x_i)\|
&\le& \min_{i=1,\ldots, k} \left\{ L \left( \gamma_i \|R_{\gamma_i}(x_i)\|\right)^\nu
+ {\cal Q} \|R_{\gamma_i}(x_i)\|\right\} \\
&\le&  L \left( \gamma_k \min_{i=1,\ldots,k}\|R_{\gamma_i}(x_i)\|\right)^\nu
+ {\cal Q} \min_{i=1,\ldots,k}\|R_{\gamma_i}(x_i)\| ,
\eeqas
which together with \eqnok{stepsize_h} and \eqnok{iter_bnd_h} then imply that
\beqas
\min_{i=1,\ldots, k} \|\phi_{\gamma_i}(x_i)\|
&\le& \frac{ L }{k^\frac{\nu}{2}}
\left[ \left(\frac{2\sqrt{2}(1+V(x_1,x^*))^\frac{1}{2}}{(2\nu)^\frac{\nu}{2} \alpha^\frac{1}{2}}\right)^\nu
+ \frac{2\sqrt{2} {\cal Q} (1+V(x_1,x^*))^\frac{1}{2}}{\alpha^\frac{2+\nu}{2}}\right]\\
&\le& \frac{2\sqrt{2} L (1+V(x_1,x^*))^\frac{1}{2}}{k^\frac{\nu}{2}}
\left[\frac{1}{(2\nu)^\frac{\nu^2}{2} \alpha^\frac{\nu}{2}} + \frac{{\cal Q}}{\alpha^\frac{2+\nu}{2}} \right]\\
&\le& \frac{2\sqrt{2} L (1+V(x_1,x^*))^\frac{1}{2}}{(\alpha k)^\frac{\nu}{2}}
\left( 4 + \frac{{\cal Q}}{\alpha}\right),
\eeqas
where the last two inequalities follow from the facts that
$\nu \in (0,1]$ and that 
\[
(2\nu)^\frac{\nu^2}{2} \ge \nu^\frac{\nu^2}{2}
\ge \nu^\frac{\nu}{2} \ge e^{-\frac{1}{2e}} \ge \frac{1}{4}.
\]

Part c) follows from \eqnok{bnd_gap} and an argument similar to the one used
in the proof of part b).
\end{proof}

\vgap

A few remarks about the results obtained in Theorem~\ref{theorem_hold}
are in place. Firstly, it seems possible to relax the assumption that
the number of iterations $k$ is given a priori. For example, if we
set
\beq \label{stepsize_h1}
\gamma_i = \frac{\alpha^\frac{1+\nu}{2}}{L (2 \nu)^{\frac{\nu}{2}}}
\left(\frac{1}{i}\right)^{\frac{1-\nu}{2}}, \ \ \ i = 1, 2, \ldots,
\eeq
we can show essentially the same rate of convergence as
the one obtained in \eqnok{iter_bnd_gap_h} when the set $X$
is bounded, and slightly worse (with an additional
logarithmic factor $\log (1/k)$) convergence rate than those
stated in \eqnok{iter_bnd_h} and \eqnok{iter_bnd_ms_h} and \eqnok{iter_bnd_gap_h}
if the set $X$ is unbounded.

Secondly, according to Theorem~\ref{theorem_hold}.b), if $F(\cdot)$ is H\"{o}lder continuous,
an $(\epsilon,0)$-strong solution of $\VI(X,F)$ can be computed
in at most
\[
{\cal O}\left\{\frac{V(x_1,x^*)^\frac{1}{\nu}}
{\alpha^\frac{\nu+2}{\nu}} \left(\frac{{\cal Q} L}{\epsilon}\right)^\frac{2}{\nu}\right\}
\]
iterations. This complexity result appears to be new also for the case
when $F(\cdot)$ is monotone and the distance generating function
$\w(\cdot) = \|\cdot\|_2^2/2$.

Thirdly, the results in Theorem~\ref{theorem_hold}
indicate how the rates of convergence of the N-EG method
depend on the continuity assumption about $F(\cdot)$.
In view of Theorem~\ref{theorem_hold}, if $F(\cdot)$ is continuous
but not necessarily H\"{o}lder continuous,
i.e., $\nu = 0$, the sequence $\{y_k\}$ generated by the N-EG method
will not converge to any solutions of $\VI(X,F)$.
We will address this issue in next section to deal with GMVI problems
with much weaker continuity assumptions on their operators.

\setcounter{equation}{0}
\section{VI problems with general continuous operators} \label{sec_adv_alg}
In this section, we consider GMVI problems where the
operator $F(\cdot)$ is continuous but not necessarily Lipschitz or H\"{o}lder
continuous. Our goal is to show that the N-EG method, after incorporating
a simple line search procedure, can generate a sequence of solutions converging to
the optimal one of $\VI(X,F)$ under this more general setting. More specifically,
we study the conditions on the continuity assumption of $F(\cdot)$ and
those on the distance generating function $\w(\cdot)$ in order to guarantee the convergence
of these types of N-EG methods applied to more general GMVI problems.

It should be noted that there exist a few earlier developments (e.g.,
Sun~\cite{Sun95-1}, Solodov and Svaiter~\cite{SolSva99-1}) that
generalized Korpelevich's extragradient method for solving GMVI problems.
However, to the best of our knowledge, there does not exist any non-Euclidean extragradient
type methods for solving GMVI problems. It should be noted that, while earlier extragradient type methods
for GMVI problems (e.g., \cite{SolSva99-1,Sun95-1}) rely on a
certain monotonicity property of the metric projection (e.g., Gafni and Bertsekas \cite{GafBer84-1}),
in general such a property is not assumed by the prox-mapping.

A related but different work to ours is due to Auslender and Teboulle~\cite{AuTe05-1}.
In~\cite{AuTe05-1}, Auslender and Teboulle studied the generalization of the mirror-prox method
for solving monotone and locally Lipschitz continuous VI problems by incorporating a line search
procedure. More specifically, they show the convergence of their methods under the assumption
that the prox-function $V(\cdot, \cdot)$ is quadratic. Our development significantly differs
from \cite{AuTe05-1} in the following aspects:
i) we deal with VI problems with generalized monotone (e.g., pseudo-monotone) operators;
ii) we use more general prox-functions which are not necessarily quadratic;
and iii) we consider VI problems with general continuous
operators, in addition to those with locally Lipschitz continuous operators.

We are now ready to describe a variant of the N-EG method obtained
by incorporating a simple linear search procedure for solving VI
problems with general continuous operators.

\vgap

\noindent{\bf The N-EG method with line search (N-EG-LS):}
\begin{itemize}
\item [] {\bf Input:} Initial point $x_1 \in X$,
initial stepsize $\gamma_0 \in (0,1)$ and $\lambda \in (0,1)$.
\item [0)] Set $k = 1$;
\item [1)] Compute $R_{\gamma_0}(x_k)$. If $\|R_{\gamma_0}(x_k)\| = 0$
{\bf terminate} the algorithm. Otherwise, choose the $\gamma_k$ with the largest value
from the list $\left\{\gamma_0, \gamma_0 \lambda, \gamma_0 \lambda^2, \ldots \right\}$
 such that
\beq \label{term_ls}
\|F(x_k) - F(y_k)\|_*^2 \le \frac{\alpha}{\gamma_k^2} V(x_k, y_k),
\eeq
where $y_k = P_{x_{k}}(\gamma_k F(x_{k}))$;
\item [3)] Compute
\beq \label{tt2-v}
 x_{k+1} = P_{x_{k}}(\gamma_k F(y_k));
\eeq
\item [4)] Set $k = k+1$ and go to Step 1.
\end{itemize}

\vgap

In order to establish the convergence of the above N-EG-LS method,
we first need to show certain properties of the line-search procedure used in
Step 1 of this algorithm. In the remaining part of this section, we
say that {\sl the line search procedure is well-defined} if
the following two conditions hold:
\begin{itemize}
\item [a)]
The line search procedure will terminate (i.e., condition
\eqnok{term_ls} will be satisfied) after a finite number of steps in
choosing $\gamma_k$ from the list $\{\gamma_0, \gamma_0 \lambda,
\ldots\}$ for any $k \ge 1$;
\item [b)] There exists $K \in \mathbb{N}$ and $\gamma^* > 0$
such that
\beq \label{min_gamma}
\gamma_k \ge \gamma^* \ \ \ \forall k \ge K.
\eeq
\end{itemize}
Note that condition b) is used to show the convergence of the
sequence $\{x_k\}$ (c.f., Lemma~\ref{asy_con}).

Traditionally, the well-definedness
of the line search procedure was established by using certain important
properties of the metric projection. Specifically, let us denote,
for any $x \in \bbr^n$ and $d \in \bbr^n$,
\beq \label{cond_trad}
\theta(\beta) := \frac{\|\Pi_X(x+ \beta d) - x\|_2}{\beta}, \ \ \
\beta > 0.
\eeq
Then, it is shown by Gafni and Bertsekas \cite{GafBer84-1} that
the function $\theta(\beta)$ is monotonically nonincreasing
with respect to $\beta$.
In Proposition~\ref{prop_linesearch}, we show that,
if $F(\cdot)$ is a general continuous operator,
the line-search procedure is well-defined under a much weaker assumption
than the aforementioned the monotonicity of $\theta(\beta)$.
Moreover, we present a sufficient condition on $\w(\cdot)$ which can
guarantee that the above assumption is satisfied. In addition, we show
in Proposition~\ref{prop_linesearch1} that we do not need any conditions similar to \eqnok{cond_trad} when
$F(\cdot)$ is locally Lipschitz continuous.
The following well-known property of
the prox-mapping (e.g., Lemma 2.1 in \cite{Nem05-1})
is used in the proof of Proposition~\ref{prop_linesearch}.

\begin{lemma} \label{pos_proxmapping}
Let $x \in X$ be given. we have
\beq \label{lip_prox}
\|P_x(\phi_1) - P_x(\phi_2)\| \le \alpha^{-1} \|\phi_1 - \phi_2\|_*, \ \
\forall \, \phi_1, \phi_2 \in \bbr^n.
\eeq
\end{lemma}

\vgap

\begin{proposition} \label{prop_linesearch}
Suppose that $F(\cdot)$ is continuous and also assume that
there exists $q > 0$ such that, for any $x\in X$ and $\phi \in \bbr^n$,
\beq \label{monotone_prox}
\frac{V(x, P_x(\gamma \phi))}{\gamma^2} \le \frac{q V(x, P_x(\beta \phi))}{\beta^2},
\ \ \, \gamma \ge \beta > 0.
\eeq
Then, the line search procedure in Step 1
of the N-EG-LS method is well-defined.
In particular, if the distance generating function $\w(\cdot)$ has
${\cal Q}$-Lipschitz continuous gradients w.r.t. $\|\cdot\|$,
then relation \eqnok{monotone_prox} holds with
$q = 1+ {\cal Q}^2/\alpha^2,
$
where $\alpha$ is the modulus of $\w(\cdot)$.
\end{proposition}

\begin{proof}
Suppose first that condition \eqnok{monotone_prox} holds.
Consider an arbitrary iteration $k$, $k \ge 1$.
Let us denote $\gamma_{kj} := \gamma_0 \lambda^j$ and
$y_{kj} := P_{x_k}(\gamma_{kj}F(x_k))$, $j \ge 0$.
Observe that $\|R_{\gamma^0}(x_k)\| > 0$ whenever the line search procedure
occurs. Using this observation, \eqnok{bnd_v} and \eqnok{def_res}, we have
\beq \label{tmp_rhs}
\frac{V(x_k, y_{k0})}{(\gamma_0)^2}  \ge \frac{\alpha}{2 (\gamma_0)^2} \|x_k - y_{k0}\|^2
= \frac{\alpha}{2} \|R_{\gamma_0}(x_k)\|^2 > 0.
\eeq
The above inequality together with \eqnok{monotone_prox} then imply that
\[
\frac{V(x_k, y_{kj})}{\gamma_{kj}^2} \ge \frac{V(x_k, y_{k0})}{q (\gamma_0)^2}
 \ge \frac{\alpha}{2 q}  \|R_{\gamma_0}(x_k)\|^2
 >0, \ \ \  \forall \, j \ge 1.
\]
Assume for contradiction that the line search procedure does not terminate
in a finite number of steps. Then, we have
\[
\|F(x_k) - F(y_{kj})\|_*^2 > \frac{\alpha  V(x_k, y_{kj})}{\gamma_{kj}^2}, \forall
j \ge 1.
\]
It then follows from the above two inequalities that
\beq \label{cond_fact}
\|F(x_k) - F(y_{kj})\|_*^2 > \frac{\alpha^2}{2 q}  \|R_{\gamma_0}(x_k)\|^2 > 0, \forall
j \ge 1.
\eeq
On the other hand, using the Lipschitz continuity of the prox-mapping (see \eqnok{lip_prox}),
and the fact that $\lim_{j \to +\infty} \gamma_{kj} = 0$, we have
$\lim_{j \to +\infty} \|x_k - y_{kj}\| = 0$. This observation, in view of
the fact that $F(\cdot)$ is continuous, then imply that
$\lim_{j \to +\infty} \|F(x_k) - F(y_{kj})\|_*^2 = 0$, which
clearly contradicts with \eqnok{cond_fact}.
Hence, the line search procedure must terminate in a finite number
of steps.

We now show that there exists $K \in \mathbb{N}$ and $\gamma^* > 0$
such that \eqnok{min_gamma} holds.
Assume for contradiction that
$\lim_{k \to +\infty}\gamma_k = 0$. Let us denote $\hat x_k := P_{x_k}(\beta^{-1}\gamma_k F(x_k))$.
By the choice of $\gamma_k$, we know that \eqnok{term_ls} is not satisfied for $y_k = \hat{x}_k$,
hence we have
\beq \label{cond_fact1}
\|F(x_k) - F(\hat x_k)\|_*^2 > \frac{\alpha}{(\beta^{-1} \gamma_k)^2} V(x_k, \hat x_k)
\ge \frac{\alpha}{q (\gamma_0)^2} V(x_k, \hat x_k)
\ge \frac{\alpha^2}{2 q (\gamma_0)^2}  \|R_{\gamma_0}(x_k)\|^2 > 0, \ \ k \ge 1.
\eeq
where the second inequality is due to \eqnok{monotone_prox}.
Using the Lipschitz continuity of the prox-mapping (see \eqnok{lip_prox}),
and the assumption that $\lim_{k \to +\infty} \gamma_{k} = 0$, we have
$\lim_{k \to +\infty} \|x_k - \hat x_k\| = 0$.
This observation, in view of
the fact that $F(\cdot)$ is continuous, then imply that
$\lim_{k \to +\infty} \|F(x_k) - F(\hat x_k)\|_*^2 = 0$, which
clearly contradicts with \eqnok{cond_fact1}.

We now show that relation \eqnok{monotone_prox} holds if $\w(\cdot)$
has ${\cal Q}$-Lipschitz continuous gradients.
Denote $x^+_\gamma \equiv P_x(\gamma \phi)$, $x^+_\beta \equiv P_x(\beta \phi)$.
It follows from Lemma \ref{tech1_prox} (with $p(\cdot) = \gamma \langle \phi, \cdot \rangle$, $\tilde x = x$ and $u^* = x^+_\gamma$)
that
\[
\gamma \langle \phi, x^+_\gamma - z \rangle + V(x, x^+_\gamma) + V(x^+_\gamma, z) \le V(x, z), \ \forall z \in X.
\]
Letting $z = x^+_\beta$ in the above relation, we have
\[
V(x, x^+_\beta) - V(x, x^+_\gamma) \ge \gamma \langle \phi, x^+_\gamma - x^+_\beta \rangle + V(x^+_\gamma, x^+_\beta),
\]
which implies that
\beqas
q \gamma^2 V(x, x^+_\beta) -  \beta^2 V(x, x^+_\gamma) &=& (q\gamma^2 -  \beta^2)  V(x, x^+_\beta)
+ \beta^2[V(x, x^+_\beta) - V(x, x^+_\gamma)]\\
&\ge& (q \gamma^2 - \beta^2)  V(x, x^+_\beta) + \beta^2 \gamma
\langle \phi, x^+_\gamma - x^+_\beta \rangle + \beta^2 V(x^+_\gamma, x^+_\beta)\\
&\ge& \left(\frac{\gamma {\cal Q}}{\alpha}\right)^2  V(x, x^+_\beta) + \beta^2 \gamma
\langle \phi, x^+_\gamma - x^+_\beta \rangle + \beta^2 V(x^+_\gamma, x^+_\beta),
\eeqas
where the last inequality follows from the definition of $q$ and the fact that $\gamma \ge \beta$.
Also note that by the optimality condition of \eqnok{s29}, we have
\[
\langle \beta \phi + \nabla \w(x^+_\beta) - \nabla \w(x), x^+_\gamma - x^+_\beta \rangle \ge 0.
\]
Combining the above two conclusions,
relations \eqnok{bnd_v} and \eqnok{smoothomega}, we obtain
\beqas
q \gamma^2 V(x, x^+_\beta) - \beta^2 V(x, x^+_\gamma)
&\ge& \left(\frac{\gamma {\cal Q}}{\alpha}\right)^2  V(x, x^+_\beta) +
\beta \gamma \langle \nabla \w(x^+_\beta) - \nabla \w(x),  x^+_\beta -x^+_\gamma \rangle
+ \beta^2 V(x^+_\gamma, x^+_\beta) \\
&\ge& \left(\frac{\gamma {\cal Q}}{\alpha}\right)^2  V(x, x^+_\beta) - \beta \gamma \|\nabla \w(x^+_\beta) - \nabla \w(x)\|_*
\|x^+_\beta -x^+_\gamma\| + \beta^2 V(x^+_\gamma, x^+_\beta)\\
&\ge&\frac{\alpha}{2} \left[
\left(\frac{\gamma {\cal Q}}{\alpha}\right)^2\|x-x^+_\beta\|^2 - 2 \beta \gamma \frac{{\cal Q}}{\alpha} \|x^+_\beta -x\|
\|x^+_\beta -x^+_\gamma\| + \beta^2 \|x^+_\beta -x^+_\gamma\|^2 \right] \\
&=&\frac{\alpha}{2} \left( \frac{\gamma {\cal Q}}{\alpha} \|x-x^+_\beta\| - \beta \|x^+_\beta -x^+_\gamma\|  \right)^2
\ge 0,
\eeqas
from which \eqnok{monotone_prox} immediately follows.
\end{proof}

\vgap

The next result identifies certain special cases where
we do not need any additional assumptions on the
$\w(\cdot)$ in order to guarantee the well-definedness of the
linear search procedure.

\begin{proposition} \label{prop_linesearch1}
Suppose that $F(\cdot)$ is locally Lipschitz continuous.
Then, regardless the choice of $\w(\cdot)$,
the line search procedure in the N-EG-LS method is well-defined.
In particular, if $F(\cdot)$ is Lipschitz continuous, then
the line search procedure will terminate in at most
\[
\max\left\{1, \log_{\frac{1}{\lambda}} \frac{\alpha}{\sqrt{2} \gamma_0 L }
\right\}
\]
steps. Moreover, in the latter case we have
\beq \label{bound_gamma_below}
\gamma_k \ge \min \left\{ \frac{\lambda \alpha}{\sqrt{2} L}, \gamma_0\right\},
\ \forall k \ge 1.
\eeq
\end{proposition}

\begin{proof} Consider the locally Lipschitz continuous case first.
Suppose for contradiction that the line search procedure is not
well-defined. Then, by \eqnok{bnd_v} and \eqnok{term_ls}, we must
have
\[
\|F(x_k) - F(y_{kj})\|_*^2 > \frac{\alpha  V(x_k, y_{kj})}{\gamma_{kj}^2}
\ge \frac{\alpha^2  \|x_k - y_{kj}\|^2}{2 \gamma_{kj}^2}, \forall
j \ge 1,
\]
which, in view of \eqnok{smooth2}, then implies that $ L^2 >
\alpha^2  / (2 \gamma_{kj}^2), j \ge 1. $ Tending $j$ to $+\infty$,
we have arrived at a contradiction. In order to show that there
exists $K \in \mathbb{N}$ and $\gamma^* > 0$ such that
\eqnok{min_gamma} holds, suppose for contradiction that $\lim_{k \to
+\infty}\gamma_k = 0$. Let us denote $\hat x_k :=
P_{x_k}(\beta^{-1}\gamma_k F(x_k))$. By the choice of $\gamma_k$,
\eqnok{term_ls} and \eqnok{bnd_v}, we have \beq \label{cond_fact2}
\|F(x_k) - F(\hat x_k)\|_*^2 > \frac{\alpha}{(\beta^{-1}
\gamma_k)^2} V(x_k, \hat x_k) \ge \frac{\alpha^2}{ 2(\beta^{-1}
\gamma_k)^2} \|x_k - \hat x_k\|^2, \ \ k \ge 1, \eeq which, in view
of \eqnok{smooth2}, then implies that $L^2 > \alpha^2 /
(2(\beta^{-1} \gamma_k)^2)$. Tending $k$ to $+\infty$, we have
arrived at a contradiction.

Now consider the Lipschitz continuous case. By \eqnok{smooth} and \eqnok{bnd_v},
we have
\[
\|F(x_k) - F(y_k)\|_*^2 \le L^2 \|x_k - y_k\|^2 \le \frac{2 L^2 V(x_k, y_k)}{\alpha}.
\]
Comparing the above inequality with \eqnok{term_ls}, we can easily show the
last part of the result.
\end{proof}

\vgap


\vgap

We are now ready to establish the main convergence properties of the
above N-EG-LS method applied to GMVI problems with a general
continuous operator $F(\cdot)$.

\begin{theorem}
Suppose that the line-search procedure is well-defined.
Then the sequences $\{x_k\}_{k \ge 1}$ and $\{y_k\}_{k \ge 1}$ generated
by the N-EG-LS method converge to a strong solution of $\VI(X,F)$.
\end{theorem}

\begin{proof}
First note that relation \eqnok{recur} still holds for the variant
of N-EG method due to our definitions of $x_k$, $y_k$ and
$\gamma_k$, $k \ge 1$. Moreover, it follows from the
well-definedness of the line search step, relation \eqnok{term_ls}
must hold for some $\gamma_k > 0$. Using relations \eqnok{recur} and
\eqnok{term_ls}, we can easily see that, for some $x^* \in X^*$,
\beq \label{con_recur} \frac{1}{2} V(x_k, y_k) \le V(x_k, x^*) -
V(x_{k+1}, x^*), \ \ \ k \ge 1. \eeq Clearly \eqnok{con_recur}
implies that the sequence $V(x_k,x^*)$ is nonincreasing. Therefore,
it converges. Moreover, the sequence $\{x_k\}$ is bounded. Summing
up the inequalities in \eqnok{con_recur}, we obtain
\[
\frac{1}{2}  \sum_{k=1}^\infty  V(x_k, y_k) \le V(x_1, x^*),
\]
which then implies that
\beq \label{lim_V}
\lim_{k \to +\infty} V(x_k, y_k) = 0.
\eeq
Using these observations, the fact that condition \eqnok{min_gamma} holds
for some $K \in \mathbb{N}$ and $\gamma^*$, and Lemma~\ref{asy_con},
there exists an accumulation point $\tilde x$ of $\{x_k\}$ such that
$\tilde x \in X^*$. We can replace $x^*$ in \eqnok{con_recur} by
$\tilde x$. Thus the sequence $\{V(x_k, \tilde x)\}$ converges.
Since $\tilde x$ is an accumulation point of $\{x_k\}$, it easily
follows that $\{V(x_k, \tilde x)\}$ converges to zero, i.e.,
$\{x^k\}$ converges to $\tilde x \in X^*$. The previous conclusion
together with \eqnok{lim_V} then imply the convergence of $\{y^k\}$.
\end{proof}

\vgap

\noindent{\bf Remark.}
Observe that we can estimate the rate of convergence of the N-EG-LS
method applied to GMVI problems with Lipschitz continuous operators.
Indeed, by using \eqnok{bnd_L} and \eqnok{bound_gamma_below}, we
have
\[
\min_{i=1, \ldots,k} \|R_{\gamma_i}(x_i)\|^2 \le
\frac{4 V(x_1, x^*)}{\alpha k \min\{\lambda^2 \alpha^2/2L, \gamma_0^2\}}, k \ge 1.
\]
The above bound is slightly worse than the one in
\eqnok{iter_bnd_l}. However, one potential advantage of the
N-EG-LS method over the N-EG method is that it does not require
the explicit input of the Lipschitz constant $L$.


\section{Numerical Results} \label{sec_num}
In this section, we report preliminary results of our
computational experiments where we compare the performance of
different variants of the N-EG method
discussed in this paper.

\subsection{Problem instances}
We focus on an important class of VI problems $\VI(X,F)$, where $X$
is the standard simplex given by $X = \{x \in \bbr^n| \sum_i x_i =
1, x_i \ge 0, i = 1, \ldots,n\}$ and $F$ is a continuous. These problems are chosen for the
following reasons: i) they have been extensively studied in the
literature (e.g.,
\cite{BenNem00,PangGab93,Sun95-1,Wat79,HarPang90}); ii) A set of complexity
results for these problems have been developed in this paper; and iii)
it is expected that the study on these VI problems with relatively simple
feasible set $X$ can shed some light on problems with more
complicated feasible set $X$.

In particular, the following instances have been used in our
numerical experiments. Note that for most of these problems, the
operator $F$ is not necessarily monotone.

\vgap

\noindent {\bf a. Kojima-Shindo (KS) problem}\\
This problem was studied in \cite{PangGab93}. The operator
$F:\bbr^4\rightarrow \bbr^4$ is defined as:
\[
F\left( x \right) = \left[ \begin{array}{l}
 3x_1^2  + 2x_1 x_2  + 2x_2^2  + x_3  + 3x_4  - 6 \\
 2x_1^2  + x_1  + x_2^2  + 10x_3  + 2x_4  - 2 \\
 3x_1^2  + x_1 x_2  + 2x_2^2  + 2x_3  + 9x_4  - 9 \\
 x_1^2  + 3x_2^2  + 2x_3  + 3x_4  - 3 \\
 \end{array} \right].
\]

\vgap

\noindent{\bf b. Watson (WAT) problem}\\
The operator $F:\bbr^{10} \to \bbr^{10}$ is given by $F(x)=Ax+b$,
where
\[
A = \left( {\begin{array}{*{20}c}
   0 & 0 & { - 1} & { - 1} & { - 1} & 1 & 1 & 0 & 1 & 1  \\
   { - 2} & { - 1} & 0 & 1 & 1 & 2 & 2 & 0 & { - 1} & 0  \\
   1 & 0 & 1 & { - 2} & { - 1} & { - 1} & 0 & 2 & 0 & 0  \\
   2 & 1 & { - 1} & 0 & 1 & 0 & { - 1} & { - 1} & { - 1} & 1  \\
   { - 2} & 0 & 1 & 1 & 0 & 2 & 2 & { - 1} & 1 & 0  \\
   { - 1} & 0 & 1 & 1 & 1 & 0 & { - 1} & 2 & 0 & 1  \\
   0 & { - 1} & 1 & 0 & 2 & { - 1} & 0 & 0 & 1 & { - 1}  \\
   0 & { - 2} & 2 & 0 & 0 & 1 & 2 & 2 & { - 1} & 0  \\
   0 & { - 1} & 0 & 2 & 2 & 1 & 1 & 1 & { - 1} & 0  \\
   2 & { - 1} & { - 1} & 0 & 1 & 0 & 0 & { - 1} & 2 & 2  \\
\end{array}} \right),
\]
$b = e_i$ and $e_i$ is the unit vector. Hence, we have $10$
different instances of Watson problem, i.e., WAT1,WAT2, \ldots, WAT10,
obtained by setting $q = e_1, e_2, \ldots, e_{10}$. This problem was studied
by Watson in \cite{Wat79}.

\vgap

\noindent{\bf c. Sun problem}\\
This problem was discussed by Sun in \cite{Sun95-1} and we consider
problems possibly in larger dimension. The operator $F:\bbr^n \to
\bbr^n$ is given by $F(x)=Ax+b$, where
\[
A = \left( {\begin{array}{*{20}c}
   1 & 2 & 2 &  \cdot  &  \cdot  &  \cdot  & 2  \\
   0 & 1 & 2 &  \cdot  &  \cdot  &  \cdot  & 2  \\
   0 & 0 & 1 &  \cdot  &  \cdot  &  \cdot  & 2  \\
    \cdot  &  \cdot  &  \cdot  &  \cdot  & {} & {} &  \cdot   \\
    \cdot  &  \cdot  &  \cdot  & {} &  \cdot  & {} &  \cdot   \\
    \cdot  &  \cdot  &  \cdot  & {} & {} &  \cdot  & 2  \\
   0 & 0 & 0 &  \cdot  &  \cdot  &  \cdot  & 1  \\
\end{array}} \right)
\]
and $b=(-1,...,-1)$. We consider problem instances with dimension
$n$ ranging from $8,000$ to $30,000.$

\vgap

\noindent{\bf d. Modified HP Hard (MHPH) problem}\\
We modify the Harker's procedure (\cite{HarPang90}) to build an
affine function $F(x)=Ax + b$, where the positive definite matrix
$A$ is randomly generated as $A=M M^T$ (hence the VI problems are
monotone). Each entry of the $n \times n$ matrix $M$ is
uniformly generated in $(-15,-12)$ and vector $b$ has been uniformly
generated in $(-500,0).$ We generated instances with dimension $n$
ranging from $1,000$ to $8,000.$

\vgap

\noindent{\bf e. Randomly generated (RG) instances}\\
We consider an affine function $F(x)=Ax+b$, where each entry of
the $n \times n$ matrix $A$ is uniformly generated in $(-50,150)$
and $q$ is uniformly generated in $(-200,300).$ We do not know if
these VI problems are monotone or not. The dimension $n$ of these
problem instances ranges from $1,000$ to $3,000$.


\subsection{Euclidean algorithms for GMVI problems}
Our first experiments are carried out to compare the two
Euclidean extragradient methods, i.e., the
E-EG and E-EG-LS method presented in this paper. We
also compare these methods with a different method for solving
pseudo-monotone VI problems developed by Sun (Algorithm C in
\cite{Sun95-1}).

Note that two parameters  $\gamma_0 \in (0,1)$ and $\lambda\in
(0,1)$ are required for the line search procedure 
in the N-EG-LS or E-EG-LS methods. We used a simple
fine-tuning procedure to determine these parameters
which is briefly described as follows: for each group of
problems, we choose a smaller set of representative instances
and run these algorithms for each pair (totally $9$ pairs)
of parameters $(\gamma_0, \lambda)$ chosen from $\{0.2, 0.4, 0.8\}$.
We terminate these algorithm until 
the value of gap function $g(\cdot)$ 
falls bellow $10^{-1}$ (our target accuracy is $10^{-3}$) and 
report the number of calls to the projection (or 
prox-mapping). 
For each algorithm, we choose a pair
$(\gamma_0, \lambda)$ corresponding to the smallest total number
of projection calls for this set of representative instances, and then use
these parameters for all the instances of the same problem. For example, 
the results of E-EG-LS method applied to MHPH problem using the aforementioned 
fine-tuning procedure are reported in Table 1.
And, in view of these results, we set $\gamma_0=0.2$ and $\lambda=0.4$ 
for the E-EG-LS method applied to all instances for the MHPH problem.

\begin{table}
\centering \label{fine-tuning}
\footnotesize
\begin{tabular}{| c| c| c|c|c| }
\hline 
\multirow{2}{*}{}Parameters & \multicolumn{4}{|c|}{Number of projection calls $np$} \\ \cline{2-5}
  \cline{2-5} $(\gamma_0, \lambda)$ & $n=1,000$ & $n=3,000$& $n=5,000$ & Total $np$ \\
\hline
  (0.2,0.2)& 1,746 & 3,074& 9,350 & 14,170 \\
\hline
 (0.2,0.4)& 2,397 & 3,069& 4,595 & 10,061 \\
  \hline
  (0.2,0.8)& 4,693 & 7,291& 13,293 & 25,277 \\
  \hline
  (0.4,0.2)& 3,307 & 3,987& 5,174 & 12,468 \\
  \hline
  (0.4,0.4)& 1,924 & 3,457& 5,811 & 11,192 \\
  \hline
  (0.4,0.8)& 5,020 & 7,794& 14,529 & 27,343 \\
  \hline
  (0.8,0.2)& 2,056 & 3,891& 4,456 & 10,403 \\
  \hline
  (0.8,0.4)& 2,384 & 3,349& 7,476 & 13,209\\
  \hline
  (0.8,0.8)& 5,342 & 8,783& 15,285 & 29,410 \\
 \hline

\end{tabular}
\caption{Fine-tuning procedure of E-EG-LS for modified HP Hard
problem}
\end{table}

Also for the Sun's algorithm, we used the parameters suggested in
\cite{Sun95-1}. All these algorithms were implemented in MATLAB
R2009b on a Core i5 3.1 Ghz computer.

We compare the number of iterations $k$, total number calls of
projection $np$ and {\it CPU time} for the above
three algorithms whenever the gap function $g(\cdot)$ evaluated at
the point $x_k$ (see \eqnok{def_gap1}) falls below $10^{-3}.$ The
results are reported for the HP hard and WAT problems as shown in Table 2 and Table 3,
while the results for other problems are similar. We conclude
from these results that using the Euclidean setup, the
performance of E-EG-LS method is comparable to Sun's method.
Moreover, E-EG-LS method can significantly outperform the E-EG
method especially when the Lipschitz constant $L$ is big. In next
subsection, we will demonstrate how we can improve the performance
of the E-EG-LS method by incorporating the non-Euclidean setup.

\begin{table}
\centering \label{hphard_lipschitz}
\footnotesize
\begin{tabular}{| c| c| c|c|c| c|c|c| c|c|c| c|c|}
\hline 
\multirow{3}{*}{}& \multicolumn{9}{|c|}{Algorithms} \\ \cline{2-10}
 & \multicolumn{3}{|c|}{E-EG}& \multicolumn{3}{|c|}{E-EG-LS$\dag$}& \multicolumn{3}{|c|}{Algorithm C}\\
 \cline{2-10} $n$ & $k$& $np$ & CPU time& $k$& $np$ &CPU time& $k$& $np$ & CPU time\\
\hline
  20& 49,502 & 99,004& 33.650 & 578 & 5,198 & 2.0592 & 2,622 & 12,198 & 4.5396 \\
\hline
  40& 9,069 & 18,138& 6.833 & 23 & 205 & 0.1092 & 118 & 515 & 0.2184 \\
 \hline
  50& 124,275 & 248,550 & 91.011 & 85& 676 & 0.2808 & 73 & 298& 0.1404 \\
  \hline
  70 & 12,911 & 25,822 & 10.203 & 20 & 193 & 0.0936 & 59 & 264&0.1404 \\
 \hline
100& 71,321 & 142,642& 285.48& 41 & 384 & 0.7956 & 61 & 287 & 0.7332  \\
\hline
  150 & 43,486&86,972 & 198.79 & 29 & 303& 0.7488& 89 & 405& 0.9204 \\
 \hline
  200 & 757,758& 1,515,516 & 3,622.1 & 504 & 5,493& 14.383& 1262 & 5,967&16.443 \\
 \hline

\multicolumn{10}{l}{%
\begin{minipage}{12cm}%
  \tiny $\dag$: We use parameters $\gamma_0=0.4, \lambda=0.4$ for E-EG-LS.
\end{minipage}}
\end{tabular}
\caption{E-EG vs. E-EG-LS for modified HP Hard problems}
\end{table}

 \begin{table}
\centering \label{smallinst}
\footnotesize
\begin{tabular}{| c| c| c|c|c| c|c|c| c|c|c| c|c|}
\hline \multirow{3}{*}{}& \multicolumn{9}{|c|}{Algorithms} \\
\cline{2-10}
 & \multicolumn{3}{|c|}{E-EG}& \multicolumn{3}{|c|}{E-EG-LS$\dag$}& \multicolumn{3}{|c|}{Algorithm C}\\
 \cline{2-10} INST & $k$& $np$ & CPU time& $k$& $np$ &CPU time& $k$& $np$ & CPU time \\
\hline
WAT1 & 81 & 162 & 0.0936& 56 & 183& 0.0936 & 43 & 183 & 0.0936  \\
 \hline
WAT2 & 25 & 50 & 0.0312& 18 & 55 & 0.0468 & 24 & 105 & 0.0936 \\
 \hline
WAT3 & - &-  &-  &-  & - & - & 548 & 2,395 & 0.8892\\
\hline
WAT4 & 94 & 188 & 0.0780& 60 & 192 & 0.0936 & 54 & 225 & 0.1092 \\
 \hline
WAT5 & 26 & 52 & 0.0312& 18 & 54 & 0.0468 & 21 & 90 & 0.0780  \\
 \hline
WAT6 & 55 & 110 & 0.0468& 37 & 113 & 0.0936 & 64 & 259 & 0.1248 \\
 \hline
WAT7 & 55 & 110 & 0.0468& 37 & 113 & 0.0624 & 95 & 342 & 0.1560  \\
 \hline
WAT8 & 53 & 104 & 0.0624& 31 & 94 & 0.0624 & 29 & 132 & 0.0780  \\
 \hline
WAT9 & 12 & 24 & 0.0312& 8 & 24 & 0.0312 & 9 & 27 & 0.0624  \\
 \hline
WAT10 & 50 & 100 & 0.0468 & 34 & 102 & 0.0624 & 28 & 125 & 0.0936  \\
 \hline
\multicolumn{10}{l}{%
 \begin{minipage}{12cm}%
   \tiny $-$:  indicates that the algorithm diverges and the instance is not a GMVI problem.\\
   \tiny $\dag$: We use parameters $\gamma_0=0.2, \lambda=0.8$ for {\sl Euclidean algorithm.}
 \end{minipage}}
\end{tabular}
\caption{E-EG vs. E-EG-LS for WAT instances}

\end{table}
\subsection{Euclidean vs. non-Euclidean algorithms for GMVI problems}
In this subsection, we conduct experiments to illustrate how one can
improve the performance of the extragradient methods by considering the
following different settings: the {\sl $p$-norm} algorithm with
$\|\cdot\| = \|\cdot\|_1$ and $\w(x) = \|x\|_p^2/2$ with $p = 1 + 1/
\ln n$, the {\sl entropy algorithm} with $\|\cdot\| = \|\cdot\|_1$
and $\w(x) = \sum_i (x_i+\delta/n) \log (x_i+\delta/n)$ with $\delta
= 10^{-16}$, as well as the {\sl Euclidean algorithm} with
$\|\cdot\|=\|\cdot\|_2$ and $\w(x) = \|x\|_2^2/2$.

We first compare these algorithms for a set of relatively smaller
problem instances (namely the KS and WAT problem). As can be seen
from Table 4, since the problem dimensions, $n=4$ for the KS
problem and $n=10$ for the WAT problem, are very small, we do not
observe significant advantages of the non-Euclidean algorithms.



\begin{table}
\centering
\footnotesize
\begin{tabular}{| c| c| c|c|c| c|c|c| c|c|c| c|c|}
\hline \multirow{3}{*}{}& \multicolumn{9}{|c|}{Algorithms} \\
\cline{2-10}
 & \multicolumn{3}{|c|}{Euclidean}& \multicolumn{3}{|c|}{p-norm}& \multicolumn{3}{|c|}{entropy}\\
 \cline{2-10} INST & $k$& $np$ & CPU time& $k$& $np$ &CPU time& $k$& $np$ & CPU time \\
\hline
KS $\dagger $ & 7  & 36 & 0.0312 & 7  & 36 & 0.0312 & 17 & 60 & 0.0312 \\
\hline
WAT1$\ddagger  $ & 56 & 183& 0.0936 & 48 & 149 & 0.0936 & 77 & 275 & 0.0468 \\
 \hline
WAT2 & 18 & 55 & 0.0468 & 20 & 60 & 0.0624& 29 & 90& 0.0156 \\
 \hline
WAT3 & - &-  &-  &-  & - & - & - &- & -\\
\hline
WAT4 & 60 & 192 & 0.0936 & 73 & 223 & 0.1248 & 34 & 102& 0.0156\\
 \hline
WAT5 & 18 & 54 & 0.0468 & 21 & 63 & 0.0624 & 38 & 114 & 0.0156 \\
 \hline
WAT6 & 37 & 113 & 0.0936 & 30 & 90 & 0.0624 & 48 & 144 & 0.0156 \\
 \hline
WAT7 & 37 & 113 & 0.0624 & 36 & 107 & 0.0780 & 44 & 132 & 0.0156 \\
 \hline
WAT8 & 31 & 94 & 0.0624 & 31 & 93 & 0.0624 & 51 & 153 & 0.0312 \\
 \hline
WAT9 & 8 & 24 & 0.0312 & 8 & 24 & 0.0312 & 14 & 42 & 0.0312 \\
 \hline
WAT10 & 34 & 102 & 0.0624 & 29 & 87 & 0.0936 & 39 & 117 & 0.0156 \\
 \hline
\multicolumn{10}{l}{%
 \begin{minipage}{12cm}%
   \tiny $-$: indicates that the algorithm diverges and the instance is not a GMVI problem.\\
    \tiny $\dagger$: We use parameters $\gamma_0=0.2, \lambda=0.4$ for {\sl Euclidean algorithm}, $\gamma_0=0.2, \lambda=0.4$ for {\sl p-norm algorithm},$\gamma_0=0.8, \lambda=0.2$ for {\sl entropy algorithm}.\\
    \tiny $\ddagger$: We use parameters $\gamma_0=0.2, \lambda=0.8$ for {\sl Euclidean algorithm}, $\gamma_0=0.2, \lambda=0.8$ for {\sl p-norm algorithm},$\gamma_0=0.8, \lambda=0.8$ for {\sl entropy algorithm}.
 \end{minipage}}
\end{tabular}
\caption{Euclidean vs. non-Euclidean for smaller instances}
\end{table}

We then consider problems of higher dimension. More specifically,
we compare these three methods applied to SUN, MHPH and RG
problems with the dimension from $1,000$ to $30,000$ and report the results
in Table 5, Table 6 and Table 7 respectively. Clearly, for many
instances, the {\it non-Euclidean algorithms} outperform
the {\it Euclidean algorithms} in terms of  the number of projection calls ($np$)
and the total CPU time. Interestingly, {{\it p-norm algorithm}} is the fastest and 
the most stable one among all these algorithms. In particular, for the MHPH instances,
the {\it p-norm algorithm} can be approximately twice faster than {\it Euclidean algorithm}.
For the RG instances, the p-norm algorithm always outperforms
the other two algorithms. In particular, it succeeds in solving
all the problem instances up to accuracy $10^{-3}$ within our iteration limit 
($100,000$ projection calls), while the other two algorithms failed 
for quite a few of these instances.

\begin{table}
\centering
\footnotesize
\begin{tabular}{| c| c| c|c|c| c|c|c| c|c|c| c|c|}
\hline \multirow{3}{*}{}& \multicolumn{9}{|c|}{Algorithms} \\
\cline{2-10}
 & \multicolumn{3}{|c|}{Euclidean$\dagger$}& \multicolumn{3}{|c|}{p-norm$\ddagger$}& \multicolumn{3}{|c|}{entropy$\S$}\\ \cline{2-10} $n$ & $k$& $np$ & CPU time& $k$& $np$ &CPU time& $k$& $np$ & CPU time\\
 \hline
 8,000 & 24 & 153 & 24.133 & 16 & 74 & 23.904 & 24 & 73 & 15.507 \\
 \hline
 10,000 & 24 & 153 & 31.840 & 17 & 79 & 31.949 & 24 & 73 & 18.502 \\
 \hline
  12,000 & 25 & 166 & 45.817 & 17 & 79 & 40.155 & 25 & 76 & 23.104 \\
 \hline
 14,000 & 26 & 178 & 59.312 & 17 & 81 & 47.315 & 25 & 76 & 28.735 \\
 \hline
 16,000 & 26 & 178 & 74.740 & 17 & 81 & 58.547 & 25 & 76 & 35.381 \\
 \hline
 18,000 & 26 & 178 & 91.167 & 17 & 81 & 65.817 & 25 & 76 & 43.852 \\
 \hline
 20,000 & 26 & 178 & 110.355 & 17 & 81 & 75.739 & 25 & 76 & 51.527 \\
 \hline
 22,000 & 26 & 178 & 132.54 & 17 & 81 & 86.175 & 26 & 79 & 61.854 \\
 \hline
 24,000 & 26 & 178 & 153.19 & 17 & 81 & 96.034 & 26 & 79 & 71.495 \\
 \hline
 26,000 & 26 & 178 & 183.08 & 17 & 81 & 108.58 & 26 & 79 & 79.217 \\
 \hline
 28,000 & 27 & 192 & 227.62 & 17 & 81 & 121.23 & 26 & 79 & 93.616 \\
 \hline
 30,000 & 27 & 192 & 251.04 & 17 & 81 & 138.51 & 26 & 79 & 104.24 \\
 \hline
 \multicolumn{10}{l}{%
  \begin{minipage}{12cm}%
    \tiny $\dagger$: {\sl Euclidean algorithm} parameters $\gamma_0=0.4, \lambda=0.4
    .$\\
    \tiny $\ddagger$: {\sl p-norm algorithm} parameters $\gamma_0=0.2, \lambda=0.4
    .$\\
    \tiny $\S$: {\sl entropy algorithm} parameters $\gamma_0=0.8, \lambda=0.8
    .$\\
  \end{minipage}}
\end{tabular}
\caption{Euclidean vs. non-Euclidean for Sun problem }
\end{table}

\begin{table}
\centering
\footnotesize
\begin{tabular}{| c|c| c|c|c| c|c|c| c|c|c| c|c|}
\hline 
\multirow{3}{*}{}& \multicolumn{9}{|c|}{Algorithms} \\ \cline{2-10}
 & \multicolumn{3}{|c|}{Euclidean $\dag$}& \multicolumn{3}{|c|}{p-norm $\ddag$}& \multicolumn{3}{|c|}{entropy$\S$}\\ \cline{2-10} $n$ & $k$& $np$ & CPU time& $k$& $np$ &CPU time& $k$& $np$ & CPU time\\
 \hline
  1,000 & 318&3,868 & 26.395 & 113 & 818& 9.969& 386 & 2,609 & 38.813 \\
   \hline
  1,500 & 523 & 6,877 & 77.657 & 108 & 822 & 16.976 & 240 & 1,640 & 58.926 \\
 \hline
  2,000 & 200& 2,519 & 41.262 & 148 & 1,147 & 28.877& 272 & 1,863 & 110.18 \\
   \hline
  2,500 & 332 & 4,418 & 110.82 & 205 & 1,604 & 49.562 & 391 & 2,774 & 257.67 \\
 \hline
  3,000 & 323& 4,341 & 165.63 & 685 & 5,350 & 246.60 & 1,189 & 8,431 & 1,189.0 \\
   \hline
  3,500 & 262 & 3,549 & 159.73 & 146 & 1,150 & 68.561 & 256 & 1,815 & 345.51 \\
 \hline
  4,000 & 304 & 4,173 & 263.77 & 521 & 4,087 & 254.72 & 905 & 6,563 & 1,576.7 \\
   \hline
  4,500 & 455 & 6,454 & 417.96 & 328 & 2,588 & 199.20 & 660 & 4,777 & 1,468.7 \\
 \hline
  5,000 & 471 & 6,730 & 529.51 & 749 & 5914 & 531.80 & 1,344 & 9,577 & 3,452.8 \\
   \hline
  5,500 & 562& 8,298 & 745.75 & 386 & 3,099 & 318.62 & 699 & 5,235 & 2,304.4 \\
 \hline
 6,000 & 495 & 7,288 & 729.18 & 463 & 3,723 & 434.43 & 737 & 5,448 & 2,901.1 \\
  \hline
  6,500 & 429& 6,327 & 704.88 & 472 & 3,837 & 509.09 & 904 & 6,426 & 3,968.6 \\
 \hline
 7,000 & 398 & 5,857 & 779.04 & 829 & 6605 & 976.66 & 1,586 & 11,587 & 8,272.6 \\
  \hline
  7,500 & 360 & 5,256 & 791.16 & 389 & 3,064 & 517.58 & 707 & 5,125 & 4,228.3 \\
 \hline
 8,000 & 699 & 10,761 & 1,769.3 & 1,248 & 10,327 & 1,866.7 & 2,969 & 22,200 & 20,543.0 \\
 \hline

\multicolumn{10}{l}{%
  \begin{minipage}{12cm}%
    \tiny $\dag$: {\sl Euclidean algorithm} parameters $\gamma_0=0.2, \lambda=0.4
    .$\\
    \tiny $\ddag$: {\sl p-norm algorithm} parameters $\gamma_0=0.2, \lambda=0.2
    .$\\
    \tiny $\S$: {\sl entropy algorithm} parameters $\gamma_0=0.8, \lambda=0.2
    .$\\
  \end{minipage}}
\end{tabular}
\caption{Euclidean vs. non-Euclidean for HP Hard problem}
\end{table}

\begin{table}
\centering \label{rg}
\footnotesize
\begin{tabular}{| c| c| c|c|c| c|c|c| c|c|c| c|c|}
\hline \multirow{3}{*}{}& \multicolumn{9}{|c|}{Algorithms} \\
\cline{2-10}
 & \multicolumn{3}{|c|}{Euclidean$\dagger $}& \multicolumn{3}{|c|}{p-norm$\ddagger $}& \multicolumn{3}{|c|}{entropy$\S$}\\ \cline{2-10} $n$ & $k$& $np$ & CPU time& $k$& $np$ &CPU time& $k$& $np$ & CPU time \\
 \hline
  1,000 & 250 & 2,001 & 13.993 & 55 & 482 & 8.7361 & 104 & 497 & 11.0605 \\
  \hline
  1,500 & 4,381 & 35,049 & 561.82 & 1,630 & 14,647 & 352.56 & $\sharp$ & $\sharp$ & $\sharp$ \\
 \hline
  2,000 & 1,627 & 14,017 & 223.88 & 791 & 6,994 & 208.48 & 3,170 & 16,084 & 822.72 \\
  \hline
  2,500 & $\sharp$ & $\sharp$ & $\sharp$ & 223 & 1,963 & 87.454 & $\sharp$ & $\sharp$& $\sharp$ \\
 \hline
 3,000 & 1,046 & 9,415 & 307.67 & 36 & 313 & 15.928 & 409& 2.070& 264.99\\
 \hline
 \multicolumn{10}{l}{%
  \begin{minipage}{12cm}%
    \tiny $\sharp$: indicates that the number of projection calls $np>100,000.$\\
    \tiny $\dag$: {\sl Euclidean algorithm} parameters $\gamma_0=0.8, \lambda=0.2
    .$\\
    \tiny $\ddag$: {\sl p-norm algorithm} parameters $\gamma_0=0.2, \lambda=0.4
    .$\\
    \tiny $\S$: {\sl entropy algorithm} parameters $\gamma_0=0.2, \lambda=0.2
    .$\\
  \end{minipage}}
\end{tabular}

\caption{Euclidean vs. non-Euclidean for randomly generated
instances}

\end{table}

\section{Conclusion}
This paper studies a class of generalized monotone variational inequality (GMVI) problems
whose operators are not necessarily monotone (e.g., pseudo-monotone) or Lipschitz continuous. Our main constribution
consists of: i) defining proper termination criterion for solving these VI problems; ii)
presenting non-Euclidean extragradient (N-EG) methods for computing 
approximate strong solutions of these problems; iii) demonstrating how the iteration complexities
of the N-EG methods depend on the global Lipschitz or H\"{o}lder continuity properties for their operators and the smoothness 
properties for the distance generating function used in the N-EG algorithms; and iv) 
introducing a variant of the N-EG algorithm by incorporating a simple line-search procedure 
to deal with problems with more general, not necessarily H\"{o}lder continuous operators. Moreover, 
numerical studies are conducted
to illustrate the significant advantages of the developed algorithms
over the existing ones for solving large-scale GMVI problems.

\bibliographystyle{plain}
\bibliography{../glan-bib}
\end{document}